%% file: AFM.tex
\documentclass[11pt]{amsart}
\usepackage[margin=1in]{geometry}
\usepackage{amsmath,amsfonts,amssymb,regexpatch}%
\usepackage[all]{xy}
\usepackage{setspace}
\usepackage{comment}
\usepackage{tikz-cd}
\usetikzlibrary{positioning}
\usepackage{microtype}
\usepackage{bm}

\usepackage{amsthm, stmaryrd}%
\usepackage{eucal}
\let\mathscr\mathcal  % Use Euler script

\usepackage{enumerate}

\usepackage{xparse, expl3}

\definecolor{darkgreen}{rgb}{0.0, 0.5, 0.0}%
\definecolor{darkblue}{rgb}{0.0, 0.0, 0.5}%
\usepackage{url}
\usepackage{hyperref}
\hypersetup{
	colorlinks = true,
	citecolor=darkblue,
	urlcolor=darkblue,
	linkcolor=darkgreen
}

\usepackage[capitalize,nosort]{cleveref}  
\crefname{lemma}{Lemma}{Lemmas}
\crefname{corollary}{Corollary}{Corollaries}
\crefname{theorem}{Theorem}{Theorems}
\crefname{equation}{Equation}{Equations}
\crefname{example}{Example}{Examples}
\crefname{section}{Section}{Sections}
\crefname{subsection}{Section}{Sections}
\crefname{claim}{Claim}{Claims}

\title{On the Computability of Cofinal \FRAISSE\ Limits}

\author{Nathanael Ackerman}
\address{Harvard University, USA}
\email{nate@aleph0.net}

\author{Cameron Freer}
\address{Massachusetts Institute of Technology, USA}
\email{freer@mit.edu}

\author{Mostafa Mirabi}
\address{Wesleyan University and The Taft School, USA}
\email{mmirabi@wesleyan.edu}


\input{macros}
\input{theorems}

\usepackage{thmtools}
\usepackage{thm-restate}
\newenvironment{claimproof}[1][Proof]{%
  \par\noindent\emph{#1.}\space\ignorespaces
}{%
  \hfill\qedsymbol\par\medskip
}
\newcommand{\lipicsEnd}{}

\theoremstyle{plain}
\newtheorem{theorem}{Theorem}[section]
\newtheorem{proposition}[theorem]{Proposition}
\newtheorem{lemma}[theorem]{Lemma}
\newtheorem{corollary}[theorem]{Corollary}
\newtheorem{claim}[theorem]{Claim}

\theoremstyle{definition}
\newtheorem{definition}[theorem]{Definition}
\newtheorem{example}[theorem]{Example}

\theoremstyle{remark}
\newtheorem{remark}[theorem]{Remark}

\begin{document}
\setlength{\leftmargini}{3em}  % Adjust the value as needed

\maketitle

%\tableofcontents

\begin{abstract}
For any collection of finite structures closed under isomorphism (i.e., an \emph{age}) which has the Hereditary Property $\HP$, the Joint Embedding Property $\JEP$, and the Cofinal Amalgamation Property $\coAP$, there is a unique (up to isomorphism) countable structure which is \emph{cofinally ultrahomogeneous} with the given age.
Such a structure is called the \emph{cofinal \Fraisse\ limit} of the age.

In this paper, we consider the computational strength needed to construct the cofinal \Fraisse\ limit of a computable age. We show that this construction can always be done using the oracle $\TuringTripleJump$, and that there are ages that require $\TuringDoubleJump$.

In contrast, we show that if one assumes the strengthening of $\coAP$ known as the Amalgamation Property $\AP$,
then the resulting limit, called the \Fraisse\ limit, can be constructed from the age using $\TuringJump$. 
Our results therefore show that the more general case of cofinal \Fraisse\ limits requires greater computational strength than \Fraisse\ limits.
\end{abstract}

%For any collection of finite structures closed under isomorphism (i.e., an age) which has the Hereditary Property (HP), the Joint Embedding Property (JEP), and the Cofinal Amalgamation Property (CAP), there is a unique (up to isomorphism) countable structure which is cofinally ultrahomogeneous with the given age. Such a structure is called the cofinal Fraïssé limit of the age.
%
%In this paper, we consider the computational strength needed to construct the cofinal Fraïssé limit of a computable age. We show that this construction can always be done using the oracle 0''', and that there are ages that require 0''.
%
%In contrast, we show that if one assumes the strengthening of (CAP) known as the Amalgamation Property (AP), then the resulting limit, called the Fraïssé limit, can be constructed from the age using 0'. Our results therefore show that the more general case of cofinal Fraïssé limits requires greater computational strength than Fraïssé limits.

\tableofcontents

%\newpage
%%%%  %%%%  %%%%  %%%% %%%%  %%%%  %%%%  %%%%
\section{Introduction}
%%%%  %%%%  %%%%  %%%% %%%%  %%%%  %%%%  %%%%

Given a collection of finite structures, an important question is to determine when 
they can
be glued together to get a ``canonical infinite limit''. In the case where the canonical limit is a structure containing all finite structures and that has the ``maximum amount of symmetry'' (namely, is \emph{ultrahomogeneous}), this question was answered by \Fraisse. Specifically, \Fraisse\ showed that such a limit, called a \Fraisse\ limit, exists if and only if the collection of structures has the Hereditary Property, the Joint Embedding Property, and the Amalgamation Property. 

\Fraisse\ limits have played a fundamental role in computer science and mathematics, including in database theory, automata theory, model theory, and ergodic theory. Further, there is often a tight connection between properties of the \Fraisse\ limit (such as having trivial definable closure,  $\aleph_0$-categoricity, or having an amenable automorphism group) and corresponding properties of the collection of finite structures used to build it (respectively, having the Strong Amalgamation Property, having only finitely many elements (up to isomorphism) of each size, or having the Ramsey property).

Given the important connection between a class of finite structures and its \Fraisse\ limit, it is worth studying how computably one can construct the \Fraisse\ limit from the collection of structures. 
This computability question was first studied by Csima, Harizanov, Miller, and Montalb\'{a}n,
who showed in \cite{Computable-Fraisse} how to computably build a \Fraisse\ limit from a computable collection of finite structures along with computable witnesses to the Hereditary Property, the Joint Embedding Property, and the Amalgamation Property. 

Of these three properties needed to build a \Fraisse\ limit from a collection of finite structures, the Amalgamation Property is the most closely related to the symmetry of the limit. However, often the Amalgamation Property is too restrictive and we instead only have the weaker version of this property called the Cofinal Amalgamation Property. The following example shows one of the ways in which the Cofinal Amalgamation Property is not as restrictive as the Amalgamation Property 
(see \cref{Definition of (AP),Definition of (coAP)}).

\begin{example}
Let $\Lang_f = \{f\}$ be the language with a single unary function symbol and let $\cK_f$ be the collection of finite $\Lang_f$-structures $\cM$ such that 
\[
\cM \models (\forall x)\,  \bigl(f(x) \neq x\bigr)
\qquad \text{and} \qquad
\cM \models (\forall x)\, \Bigl(\bigl(f(f(x)) = x\bigr) \Or \bigl(f(f(f(x))) = x\bigr)\Bigr).
\]
Note that $\cK_f$ is a \Fraisse\ class. 

Let $\Lang_R = \{R\}$ be the language with a single binary relation symbol and let $\cK_R$ be the collection of $\Lang_R$-structures $\cM$ where $R^{\cM}$ is the graph of a function $r$ such that $\cM \models (\forall x)\, \bigl(r(x) \neq x\bigr)$ and $\cM \models (\forall x)\, \Bigl(\bigl(r(r(x)) = x\bigr) \Or \bigl(r(r(r(x))) = x\bigr)\Bigr)$. Note that $\cK_R$ does not have the Amalgamation Property. To see this, suppose $\cM_2$ is the unique element of $\cK_R$ with underlying set $\{x, y\}$ and $\cM_3$ is the unique element of $\cK_R$ with underlying set $\{x,y, z\}$ and $r^{\cM_3}(x) = y$ and $r^{\cM_3}(y) = z$. If $\cN$ were the result of amalgamating $\cM_2$ and $\cM_3$ over $\{x\}$ then we would have $r^{\cN}(r^{\cN}(x)) = x$ and $r^{\cN}(r^{\cN}(r^{\cN}(x))) = x$, implying $r^{\cN}(x) = x$.  However, $\cK_{R}$ does have the Cofinal Amalgamation Property, as it has amalgamation over all subsets that are closed under $r$. 
\end{example}

When studying structures with functions, one often does not care if the function is represented via a function symbol or via a relation symbol encoding the graph of the function. As this example shows, 
the Cofinal Amalgamation Property, 
unlike the Amalgamation Property, 
is not sensitive to the specific way that functions are represented in a structure.

Classes of structures with the Cofinal Amalgamation Property
also admit 
limit objects (called \emph{cofinal \Fraisse\ limits}), but such limit objects only satisfy the weaker notion of \emph{cofinal ultrahomogeneity}. In this paper, we study how computationally difficult it is to construct a cofinal \Fraisse\ limit from a collection of finite structures satisfying the Hereditary Property, the Joint Embedding Property, and the Cofinal Amalgamation Property.

%%%%  %%%%  %%%%  %%%% %%%%  %%%%  %%%%  %%%%
\subsection{Related Work}
%%%%  %%%%  %%%%  %%%% %%%%  %%%%  %%%%  %%%%

\Fraisse\ limits  \cite{MR0069239}, the Amalgamation Property $\AP$, and ultrahomogeneity are key concepts in classical model theory \cite[Chapter~7]{MR1221741}, which describe an important way in which certain countable structures can be built as a generic limit of an appropriate collection of finite structures \cite{MR1141931}.
In recent years, deep connections between \Fraisse\ limits and ergodic theory have also been discovered
\cite{MR3274785,MR2140630}.

\Fraisse\ limits and their computability have been considered in the context of verification of database systems \cite{DBLP:conf/pods/BojanczykST13}, and researchers have investigated which \Fraisse\ limits can be described by automata \cite{DBLP:conf/lics/KhoussainovNRS04}. Within computer science, \Fraisse\ limits have also been studied in 
the context of denotational semantics for programming languages 
\cite{DBLP:conf/lics/DrosteG90, MR1224221}, 
fuzzy logic \cite{MR3810321}, and 
the complexity of constraint satisfaction problems in AI \cite{MR2075214} and in phylogenetics and computational linguistics \cite{MR3689377}.

In this paper, we primarily study the weaker notions of cofinal \Fraisse\ limits, the Cofinal Amalgamation Property $\coAP$, and cofinal ultrahomogeneity.
Cofinal ultrahomogeneity was introduced by Calais \cite{MR0232725, MR0233739} (under the name ``pseudo-homogeneity'') along with the corresponding amalgamation property and limit objects, and was rediscovered by Truss \cite{MR1162490} in the context of generic automorphisms. For more on the history of these notions and their relationship to other notions of limit, amalgamation, and ultrahomogeneity, 
see \cite{MR4292067} and \cite{MR4433330}.
These other notions include the even weaker notions of weak \Fraisse\ limits, the Weak Amalgamation Property $\WAP$, and weak ultrahomogeneity, 
which do generalize the cofinal concepts, but do not yet include many more known examples.
For a study of the computability of the connections between 
weak \Fraisse\ limits, $\WAP$, and weak ultrahomogeneity, 
see \cite{futureworkAFM}, which generalizes the results of the current paper
(often requiring substantially more complicated proofs).

For more on the Cofinal Amalgamation Property including several examples, see
\cite{KruckmanNotes} and
\cite[Section~3.1]{MR3697592}; for a categorical presentation, see
\cite{MR4369354}.
The Cofinal Amalgamation Property has been used in the study of
topological dynamics
\cite{MR2308230}
and the Ramsey property
\cite{bartos2021weak,MR4478610}, and
has also been studied in \cite{MR4452123}, \cite{MR4458208}, and \cite{drzewiecka2023generics}.

The computability of the \Fraisse\ construction was first studied in \cite{Computable-Fraisse}, and we build on several of their results. Isomorphisms between computable \Fraisse\ limits were also studied in \cite{MR3016251} and \cite{MR3944680}, 
and a related computability notion was considered in \cite{MR3722988,CenzerAdamsNg}.
For more on the computability of certain specific \Fraisse\ limits, see \cite{MR3034686}.

%%%%  %%%%  %%%%  %%%% %%%%  %%%%  %%%%  %%%%
\section{Summary of Main Results}
%%%%  %%%%  %%%%  %%%% %%%%  %%%%  %%%%  %%%%

Given an \emph{age}, i.e., a collection of finite structures closed under isomorphisms,\footnote{Note that some of the model theory literature refers only to the \emph{age of a structure}, in which case an age always has the Hereditary Property and Joint Embedding Property. In this work we use the more general notion, in part so that the definition of \emph{computable age} is computably checkable.} one is often interested in a countable structure which can be thought of as the limit of the age, i.e., a countable structure whose age is the given one, and where the countable structure is generic in some sense. 

In this paper, we will be interested in the computable content of going from an age to its corresponding limit. It turns out that our results on the computable content of these constructions will primarily be in terms of the \emph{embedding information} of the age. The embedding information of a computable age $\compcK$, which we denote $\EmbedInfo(\compcK)$, describes when a map between two elements of $\compcK$ can be extended to an embedding. We observe in \cref{Embdedding info is computable from 0'} that $\EmbedInfo(\compcK)$ is always computable from the Turing jump of $\compcK$, and is computable in the case
where the language of $\compcK$ is finite and relational.

The most common instance of the phenomenon of an age having a \emph{limit} is that of the \Fraisse\ limit. Given an age (with countably many isomorphism classes) satisfying the Hereditary Property $\HP$, the Joint Embedding Property $\JEP$, and the Amalgamation Property $\AP$, there is a unique structure, called the \Fraisse\ limit, whose age is the given one and which is ultrahomogeneous. In \cite{Computable-Fraisse} it was shown how to construct a \Fraisse\ limit from computable witnesses for $\HP$, $\JEP$, and $\AP$. In \cref{HP implies 0'-computable HP}, \cref{(JEP) implies 0' computable (JEP)}, and \cref{(AP) implies 0' computable (AP)} we show that such witnesses are always computable from the embedding information. Putting this together, we obtain the following result.

\begin{corollary*}[\cref{Bound on the computability of Fraisse limit}]
Suppose $\compcK$ is a computable age with $\HP$, $\JEP$, and $\AP$. Then $\compcK$ has an $\EmbedInfo(\compcK)$-computable \Fraisse\ limit. 
\end{corollary*}

While \Fraisse\ limits are the most common type of limit of an age, often we have ages which satisfy $\HP$ and $\JEP$, but which do not satisfy $\AP$, and instead merely satisfy the weaker notion of the Cofinal Amalgamation Property $\coAP$. In this case the age still has a type of limit, called the cofinal \Fraisse\ limit, but the limiting object no longer need be ultrahomogeneous and is merely required to satisfy the weaker notion of cofinal ultrahomogeneity. 

As with the case of \Fraisse\ limits, we show that from computable witnesses for $\HP$, $\JEP$, and $\coAP$, we can construct the cofinal \Fraisse\ limit. 
We frame this result in terms of (relativized) computable versions of the corresponding properties. (For example, $\CHP(\StandardTuringDegree)$ denotes the $\StandardTuringDegree$-computable version of $\HP$.)

\begin{theorem*}[\cref{prop:computable-cofinal-Fr-limit}]
Let $\StandardTuringDegree$ be a Turing degree.
Suppose  
\begin{itemize}

\item[(a)] $\compcK$ is an $\StandardTuringDegree$-computable age,  

\item[(b)] $\TuringDegree[\EmbedInfo(\compcK)] \Turingleq \StandardTuringDegree$, and

\item[(c)] $\compcK$ has $\CHP(\StandardTuringDegree)$, $\CJEP(\StandardTuringDegree)$, and $\CcoAP(\StandardTuringDegree)$. 
\end{itemize}
Then there is an $\StandardTuringDegree$-computable cofinal \Fraisse\ limit of $\compcK$. 
\end{theorem*}

With this result in hand, we reduce the problem of computing the cofinal \Fraisse\ limit from an age to the problem of computing witnesses to $\coAP$ from an age (as we have already shown that witnesses to $\HP$ and $\JEP$ are computable from the embedding information).
We are able to show the following upper bounds
on the computability of $\coAP$ and of cofinal \Fraisse\ limits.

\begin{theorem*}[\cref{Upper bounds on the computability of (CAP)} (b)]
If $\compcK$ is a computable age with $\coAP$, then it 
has $\CcoAP(\TuringDoubleJump[\EmbedInfo(\compcK)])$.
\end{theorem*}

\begin{corollary*}[\cref{Bound on the computability of cofinal Fraisse limit}]
Suppose $\compcK$ is a computable age with $\HP$, $\JEP$, and $\coAP$. Then $\compcK$ has an $\TuringDoubleJump[\EmbedInfo(\compcK)]$-computable cofinal \Fraisse\ limit. 
\end{corollary*}

We then turn our attention to the question of lower bounds for the computability of $\coAP$. We consider two situations, one where the language is finite and relational (and hence the embedding information is computable), and one where we allow the language to be infinite and to have function symbols.
We obtain the following lower bound on the computability of $\coAP$ for finite relational languages.

\begin{theorem*}[\cref{Lower bound on coAP in finite language}]
There is a computable age $\compcK$, which is the canonical computable age of some structure in a finite relational language, such that $\compcK$ has $\coAP$ but if it has $\CcoAP(\StandardTuringDegree)$ 
for some Turing degree $\StandardTuringDegree$,
then $\TuringJump \Turingleq \StandardTuringDegree$. 
\end{theorem*}

We also obtain the corresponding lower bound on the computability of cofinal \Fraisse\ limits for finite relational languages.

\begin{corollary*}[\cref{Lower bound on cofinal Fraisse limit in finite language}]
There is a computable age $\compcK$, which is the canonical computable age of some structure in a finite relational language, such that if $\StandardTuringDegree$ is a Turing degree and $\cM$ is 
an $\StandardTuringDegree$-computable structure that is $\StandardTuringDegree$-computably cofinally ultrahomogeneous
and a cofinal \Fraisse\ limit of $\compcK$, then $\TuringJump \Turingleq \StandardTuringDegree$. 
\end{corollary*}

Finally, in the general case, we obtain lower bounds on the computability of $\coAP$ and of cofinal \Fraisse\ limits. 
\begin{theorem*}[\cref{Lower bound on coAP in infinite language}]
There is a computable age $\compcK$, which is the canonical computable age of some structure,
such that $\compcK$ has $\coAP$ but if it has $\CcoAP(\StandardTuringDegree)$ 
for some Turing degree $\StandardTuringDegree$,
then $\TuringDoubleJump \Turingleq \StandardTuringDegree$. 
\end{theorem*}
\begin{corollary*}[\cref{Lower bound on cofinal Fraisse limit in infinite language}]
There is a computable age $\compcK$, which is the canonical computable age of some structure, such that if $\StandardTuringDegree$ is a Turing degree and $\cM$ is 
an $\StandardTuringDegree$-computable structure that is $\StandardTuringDegree$-computably cofinally ultrahomogeneous
and a cofinal \Fraisse\ limit of $\compcK$,
then $\TuringDoubleJump \Turingleq \StandardTuringDegree$. 
\end{corollary*}

%%%%  %%%%  %%%%  %%%%  %%%%  %%%%
\section{Basic Model Theory and Notation}
%%%%  %%%%  %%%%  %%%%  %%%%  %%%%

Throughout this paper $\Lang$ will be a (not necessarily relational) language. All relation symbols will have positive arity, but we will allow function symbols to have arity $0$ (and we will treat $0$-ary function symbols as constant symbols). If $\cA$ and $\cB$ are $\Lang$-structures, we write $\cA \subseteq \cB$ to denote the fact that $\cA$ is a substructure of $\cB$. For a finite sequence $\aa$ and a set $A$, we write $\aa \subseteq A$ to mean that each element of $\aa$ is in $A$. When $A$ is the underlying set of a structure $\cA$ we will use the standard model theory notation of writing $\aa \in \cA$ for $\aa \subseteq A$. We use the term \emph{substructure} in the standard model-theoretic sense, i.e., to mean an induced substructure that preserves whether or not each relation holds.

By an \defn{$\Lang$-tuple} we will mean a pair $(\aa, \cA)$ where $\cA$ is an $\Lang$-structure and $\aa \in \cA$. We will often abuse notation and refer to an $\Lang$-tuple $(\aa, \cA)$ by $\aa$ when the background structure $\cA$ is clear. 
If $P$ is a property of sequences we say that $P$ holds of an $\Lang$-tuple $(\aa, \cA)$ when it holds of $\aa$. (For example, the length of $(\aa, \cA)$ is the length of $\aa$.)

If $\cM$ is a structure and $A$ is a finite subset of $\cM$, we define $\Closure<\cM>(A)$ to be the smallest substructure of $\cM$ that contains $A$ (i.e., the closure of $A$ under all $\Lang$-terms). Suppose $(\aa, \cM)$ and $(\bb, \cN)$ are $\Lang$-tuples of the same length; we say $(\aa, \cM) \tuplesim (\bb, \cN)$ if for every atomic formula $\psi$ in the empty language, 
%\[
$\cM \models \psi(\aa)$ 
if and only if
$\cN \models \psi(\bb)$.
%\]     
This is equivalent to saying that whenever two coordinates of $\aa$ are equal (or not equal) the corresponding coordinates of $\bb$ are equal (or not equal, respectively). 

We say $(\aa, \cM) \clsim (\bb, \cN)$ if for every atomic formula $\varphi$, we have
%\[
$\cM \models \varphi(\aa)$ 
%\quadiff 
if and only if
$\cN \models \varphi(\bb)$.
%\]
Note that this is equivalent to the statement that the map taking the tuple $\aa$ to $\bb$ can be extended to an isomorphism between $\Closure<\cM>(\aa)$ and $\Closure<\cN>(\bb)$. In this context we refer to this isomorphism by $\ClosureMap[\aa](\bb)$.

Note that if $\aa \in \cM_0$ with $\cM_0$ a substructure of $\cM$ and $\bb \in \cN_0$ with $\cN_0$ a substructure of $\cN$, then $(\aa, \cM_0) \clsim (\bb, \cN_0)$ if and only if $(\aa, \cM) \clsim (\bb, \cN)$. We will therefore often refer to this relationship simply as $\aa \clsim \bb$ (and similarly for $\aa \tuplesim \bb$).

We say a collection of $\Lang$-structures $\cK$ is \defn{uniformly finite} if $\Lang$ is a finite language and there is a computable function $f\:\w \to \w$ such that for all $\cM \in \cK$ and $\aa \in \cM$, we have $|\Closure<\cM>(\aa)| \leq f(|\aa|)$. We say that $\cM$ is \defn{uniformly finite} if $\{\cM\}$ is uniformly finite.

We write $\TuringJump$ for the Turing jump of the minimal Turing degree, $\TuringDegreeZero$. We  write $\TuringDegree[a]\Turingleq \TuringDegree[b]$ to denote that the Turing degree $\TuringDegree[a]$ is Turing reducible to the Turing degree $\TuringDegree[b]$. Fix a standard enumeration of Turing machines and let $\{e\}(n)$ be the result of running the $e$th Turing machine on input $n\in \w$. We write 
$\{e\}(n) \halts$ when this computation halts, and 
$\{e\}(n) \nohalts$ when it does not halt.

If $\sigma$ is a sequence and $a$ is an element, we write $\sigma\^ a$ to be the sequence obtained by appending $a$ to the end of $\sigma$. If $\sigma$ and $\tau$ are sequences, a map $f\:\sigma \to\tau$ is a function which takes as input a pair $(a, i)$ where $a = \sigma(i)$, and returns as output a pair $(b, j)$ where $b = \tau(j)$. 

If $f$ is a function which outputs $k$-sequences we write $f(x) = (f_0(x), \dots, f_{k-1}(x))$, so that each $f_i$ is a function that outputs the $i$th element of the tuple on a given input. For an arbitrary function $f\:X \to Y$ and a subset $X_0 \subseteq X$, we let $f``(X_0)$ denote the image of $X_0$ under $f$, i.e., $\{f(x) \st x \in X_0\}$.  
We will often work with sequences of tuples, and will write, e.g., $(a_i, b_i)_{i\in\w}$ to mean the sequence of pairs $\bigl( (a_i, b_i) \bigr)_{i\in\w}$.

\input{ingredients/ComputableRepresentationsSection.tex}

%%%%  %%%%  %%%%  %%%%  %%%%  %%%%  %%%%  %%%%
\section{Computable Ages}
\label{sec:comp-ages}
%%%%  %%%%  %%%%  %%%%  %%%%  %%%%  %%%%  %%%%
\subsection{Basic Definitions of Computable Ages}
%%%%  %%%%  %%%%  %%%%  %%%%  %%%%  %%%%  %%%%

We now introduce the notion of a computable age. 

\begin{definition}
\label{Definition age}
An \defn{age} for $\Lang$ is a collection of finitely generated structures closed under isomorphism.

If $\Lang$ is a language with a computable representation, a \defn{computable representation} of an age $\cK$ is a
computably enumerable set
$\compcK$ consisting of objects of the form $(\aa, \cA, i)$ where
\begin{itemize}
\item for each $i \in \w$ there is a unique element of the form $(\aa, \cA, i) \in \compcK$,
\item  for each $(\aa, \cA, i) \in \compcK$, we have that
%\begin{itemize}
%\item 
(I) $i\in \w$;
(II)
$(\aa, \cA)$ is an $\Lang$-tuple such that $\cA = \Closure<\cA>(\aa)$; 
(III)
%\item 
$\cA$ is a computable $\Lang$-structure (with respect to the computable representation of $\Lang$);  and
(IV)
%\item 
$\cA\in \cK$;  and
%\end{itemize}

\item for every $\cB \in \cK$ there is an $(\aa, \cA, i) \in \compcK$ such that $\cB \cong \cA$. 

\end{itemize}

We will refer to a computable representation of an age as a \defn{computable age}. We will omit mention of $\cK$ when it is clear from context. 

If $\compcK$ is a computable age and $i \in \w$, we 
write $(\AgeTuple[\compcK](i), \AgeStr[\compcK](i))$ to denote the unique $\Lang$-tuple such that
$(\AgeTuple[\compcK](i), \AgeStr[\compcK](i), i) \in \compcK$, and define
$\AgeIndex[\compcK](i) = 
(\AgeTuple[\compcK](i), \AgeStr[\compcK](i), i)$. We omit the subscript $\compcK$ from 
$\mathbb{I}_{\compcK}$
when it is clear from context.
We will abuse notation and write $\cA \in \compcK$ if $\cA = \AgeStr[\compcK](i)$ for some $i \in \w$. Similarly we write $(\aa, \cA) \in \compcK$ if $(\aa, \cA) = (\AgeTuple[\compcK](i), \AgeStr[\compcK](i))$ for some $i \in \w$. 
\end{definition}

\begin{definition}   
A \defn{potential embedding} in $\compcK$ is a triple of the form $(\AgeIndex(i), \AgeIndex(j), \cc)$ where $i,j \in \w$ and $\cc \in \AgeStr[\compcK](j)$. We write $\Embedding<\AgeIndex(i)>[\AgeIndex(j)](
\cc) = (\AgeIndex(i), \AgeIndex(j), \cc)$. Define the \defn{domain} of $\Embedding<\AgeIndex(i)>[\AgeIndex(j)](\cc)$ to be $\dom(\Embedding<\AgeIndex(i)>[\AgeIndex(j)](\cc)) = \AgeIndex(i)$, the \defn{codomain} to be 
$\codom(\Embedding<\AgeIndex(i)>[\AgeIndex(j)](\cc)) = \AgeIndex(j)$ and the \defn{range} of $\Embedding<\AgeIndex(i)>[\AgeIndex(j)](\cc)$ to be $\range(\Embedding<\AgeIndex(i)>[\AgeIndex(j)](\cc)) = \cc$. In diagrams we will write $\cA \dashedarrow{\cc}_{\compcK}\cB$ to signify that $(\cA, \cB, \cc)$ is a potential embedding in $\compcK$. 

A potential embedding $(\AgeIndex(i), \AgeIndex(j), \cc)$ in $\compcK$ is an \defn{embedding} if $\AgeTuple[\compcK](i) \clsim \cc$. We will write $\cA \xrightarrow{\cc}_{\compcK} \cB$ to signify that $(\cA, \cB, \cc)$ is an embedding in $\compcK$. 
\end{definition}

Note that the collection of potential embeddings is a computable set. However, the collection of embeddings is, in general, only co-c.e.  For this reason it will be important to use potential embeddings when building a cofinally ultrahomogeneous structure from its age.

\begin{definition}  
Let $F$ and $G$ be potential embeddings in $\compcK$ and suppose $\codom(F) = \dom(G)$. 
When $G$ is not an embedding, we define 
$G \circ F = \bigl(\dom(F), \codom(G), \range(G)\bigr)$.
When $G$ is an embedding, we define 
$G \circ F = \bigl(\dom(F), \codom(G), \ClosureMap[\AgeTuple(j)](\range(G))(\range(F)\bigr)$, 
where $j \in \w$ is such that $\AgeIndex(j) = \dom(G)$. 
We write $\id_{\AgeIndex(i)}$ to denote the triple $(\AgeIndex(i), \AgeIndex(i), \AgeTuple(i))$. 
\end{definition}

One can check that this notation gives rise to
a category $\CAT(\compcK)$ whose objects are elements of $\compcK$, and where the maps between $\cA$ and $\cB$ are those potential embeddings $F$ with $\dom(F) = \cA$ and $\codom(F) = \cB$.

The following defines an important set associated to any pair of computable ages, as it will let us tell when 
a potential embedding between  two substructures of an age is in fact an embedding.

\begin{definition}
Suppose $\compcK$ is a computable age. Define the \defn{embedding information} related to $\compcK$, denoted $\EmbedInfo(\compcK)$, to be the collection of tuples $(\AgeIndex[\compcK](k_0), \bb_0, \AgeIndex[\compcK](k_1), \bb_1)$ such that 
%$\begin{itemize} 
%\item 
$\bb_0 \in \AgeStr[\compcK](k_0)$ and
$\bb_1 \in \AgeStr[\compcK](k_1)$, and
%
%\item 
$\bb_0 \clsim \bb_1$. 
%\end{itemize}
\end{definition}

We will see that, given two computable cofinally ultrahomogeneous structures that are isomorphic, one can always extend any partial isomorphism (between appropriate elements of the age) to an isomorphism. However, in order to build a \emph{computable} such isomorphism, we will need to be able to computably extend various partial isomorphisms; we will use the embedding information to do so.
\begin{lemma}
\label{Embdedding info is computable from 0'}
If $\compcK$ is a computable age, then $\EmbedInfo(\compcK) \Turingleq \TuringJump$.  
Further, if $\compcK$ is a uniformly finite computable age, then $\EmbedInfo(\compcK)$ is computable. 
\end{lemma}
\begin{proof}
The atomic diagram of the closure of a tuple is c.e., and 
therefore it is c.e.\ to determine when a map does not extend to a partial embedding, establishing the first claim. The second claim is immediate from the definition.
\end{proof}

\begin{definition}

Let $\compcK_0$ and $\compcK_1$ be computable ages,
and let $\StandardTuringDegree$ be a Turing degree. 
Define an \defn{isomorphism} from $\compcK_0$ to $\compcK_1$ to be a pair of maps $\bm{\alpha}= (\alpha_0, \alpha_1)$ such that
\begin{itemize}
\item $\alpha_0\:\compcK_0 \to \compcK_1$ and $\alpha_1\:\compcK_1 \to \compcK_0$, 
\item for all $(\aa, \cA, i) \in \compcK_0$ with $\alpha_0(\aa, \cA, i) = (\bb, \cB, j)$ we have $\aa \clsim \bb$, and

\item for all $(\bb, \cB, j) \in \compcK_1$ with $\alpha_1(\bb, \cB, j) = (\aa, \cA, i)$ we have $\aa \clsim \bb$. 
\end{itemize}
The pair $\bm{\alpha}$ is \defn{$\StandardTuringDegree$-computable} when the maps
$\alpha_0^*\: \w \to \w$ 
and
$\alpha_1^*\: \w \to \w$ that satisfy 
\begin{itemize}
\item $\alpha_0(\AgeIndex[\compcK_0](i)) = \AgeIndex[\compcK_1](\alpha^*_0(i))$ for $i\in \w$ and 

\item $\alpha_1(\AgeIndex[\compcK_1](i)) = \AgeIndex[\compcK_0](\alpha^*_1(i))$ for $i\in \w$

\end{itemize}
are $\StandardTuringDegree$-computable. 

We say that  $\compcK_0$ and $\compcK_1$ are \defn{$\StandardTuringDegree$-computably isomorphic} when there is some $\StandardTuringDegree$-computable isomorphism between them. When $\StandardTuringDegree = \bm{0}$ we say that they are \defn{computably isomorphic}.

Define $\bm{\alpha}^{-1} = (\alpha_1, \alpha_0)$. Note that $\bm{\alpha}^{-1}$ is an isomorphism from $\compcK_1$ to $\compcK_0$. 
\end{definition}

Note that
two computable ages $\compcK_0$ and $\compcK_1$
are isomorphic if and only if they are both computable representations of the same age, 
which holds if and only if for some Turing degree $\StandardTuringDegree$ they are $\StandardTuringDegree$-isomorphic.

\begin{remark}
When $i$ and $j$ are clear from context we will abuse notation and write 
$\alpha_0(\cA)$ or 
$\alpha_0(\aa, \cA)$ to mean $\alpha_0(\aa, \cA, i)$, 
and $\alpha_1(\cB)$ or $\alpha_1(\bb, \cB)$ to mean $\alpha_1(\bb, \cB, j)$.
Likewise, for 
$\dd \in \cA$ and $(\aa, \cA, i) \in \compcK_0$ we will
write $\alpha_0(\dd, \aa)$ to mean the pair
$\bigl(\ClosureMap[\aa](\alpha_0(\aa))(\dd),\, \alpha_0(\aa)\bigr)$, 
and
for $\dd \in \cB$ and $(\bb, \cB, j) \in \compcK_1$ we will
write $\alpha_1(\dd, \bb)$ to mean
$\bigl(\ClosureMap[\bb](\alpha_1(\bb))(\dd),\, \alpha_1(\bb)\bigr)$. 
\end{remark}

We now define how to apply an isomorphism to a potential embedding.

\begin{definition}
Let $\compcK_0$ and $\compcK_1$ be computable ages and let $\bm{\alpha}=(\alpha_0, \alpha_1)$ be an isomorphism from $\compcK_0$ to $\compcK_1$. 
Suppose $\Embedding<\AgeIndex(i)>[\AgeIndex(j)](\cc)$ is a potential embedding in $\compcK_0$.
Define the application of $\bm{\alpha}$ to $\Embedding<\AgeIndex(i)>[\AgeIndex(j)](\cc)$ by 
%\[
$\bm{\alpha}(\Embedding<\AgeIndex(i)>[\AgeIndex(j)](\cc)) =
\bigl(\alpha_0(\AgeIndex(i)),\,
\alpha_0(\AgeIndex(j)),\,
\ClosureMap[\AgeTuple(j)](\alpha_0(\AgeTuple(j)))(\cc)\bigr)$.
%\]
\end{definition}

\begin{lemma}
Suppose $\bm{\alpha} = (\alpha_0, \alpha_1)$ is an isomorphism between computable ages $\compcK_0$ and $\compcK_1$. 
Then $\bm{\alpha}$ extends to an equivalence of categories between $\CAT(\compcK_0)$ and $\CAT(\compcK_1)$ with inverse equivalence $\bm{\alpha}^{-1}$. 
\end{lemma}  
\begin{proof}
If $\Embedding<\AgeIndex(i)>[\AgeIndex(j)](\cc)$ is a potential embedding, then $\bm{\alpha}(\Embedding<\AgeIndex(i)>[\AgeIndex(j)](\cc))$ is also a potential embedding. Further $\Embedding<\AgeIndex(i)>[\AgeIndex(j)](\cc)$ is an embedding if and only if $\bm{\alpha}(\Embedding<\AgeIndex(i)>[\AgeIndex(j)](\cc))$ is. 
\end{proof}

Given a computable structure, the collection of its finitely generated substructures, appropriately enumerated, forms a computable age.

\begin{definition}
Let $\cM$ be an $\Lang$-structure. The \defn{age of $\cM$} is the collection of all $\Lang$-structures isomorphic to a finitely generated substructure of $\cM$. 

Suppose $\cM$ is computable, and let $(\aa_i)_{i \in \w}$ be some computable enumeration of all finite tuples in $\cM$ in which every tuple appears infinitely often.
We say that the sequence $(\aa_i,\, \Closure<\cM>(\aa_i),\, i)_{i \in \w}$ is a \defn{canonical computable age} of $\cM$. Let $\compcK[\cM]$ be one such canonical computable age.  
A computable age $\compcK$ is \defn{computably the age of} $\cM$ if it is computably isomorphic to $\compcK[\cM]$.
\end{definition}

The following basic facts are immediate from the definition.

\begin{lemma}
Let $\cM$ be a computable $\Lang$-structure.

\begin{itemize}
\item Any two canonical computable ages of $\cM$ are computably isomorphic. 

\item The canonical computable age of $\cM$ is computably the age of $\cM$. 

\item If $\cK$ is the age of $\cM$, and $\compcK$ is computably the age of $\cM$, then $\compcK$ is a computable representation of $\cK$. 
\end{itemize}
\end{lemma}

Note that the paper \cite{Computable-Fraisse} refers to the  canonical computable age of $\cM$ as ``the canonical representation of the age of $\cM$'', and speaks of ``a canonical representation of the age of $\cM$'' to mean any computable representation that is computably isomorphic to the canonical one.

\begin{lemma}
\label{Isomorphism of computable structures lifts to isomorphism of their canonical ages}
For each $i \in \{0, 1\}$, let 
$\cM_i$ be a computable structure and suppose that $\compcK_i$ is computably the age of $\cM_i$. If $\cM_0$ and $\cM_1$ are $\StandardTuringDegree$-computably isomorphic, then $\compcK_0$ and $\compcK_1$ are $\StandardTuringDegree$-computably isomorphic as well. 
\end{lemma}

\begin{proof}
Suppose $f\:\cM_0 \to \cM_1$ is an $\StandardTuringDegree$-computable isomorphism between $\cM_0$ and $\cM_1$. 
It suffices to show that $\compcK[\cM_0]$ is $\StandardTuringDegree$-computably isomorphic to $\compcK[\cM_1]$. 
Let $\alpha_0\:\compcK[\cM_0] \to \compcK[\cM_1]$ be the map where $\alpha_0(\aa, \cA, i) = (f(\aa), \Closure<\cM_1>(f(\aa)), j)$,
where $j$ is the first index such that $(f(\aa), \Closure<\cM_1>(f(\aa)), j) \in \compcK[\cM_1]$. Similarly, let 
%\[
$\alpha_1\:\compcK[\cM_1] \to \compcK[\cM_0]$
%\]
be the map defined by 
%\[
$\alpha_1(\aa, \cA, i) = (f^{-1}(\aa), \Closure<\cM_0>(f^{-1}(\aa)), j)$,
%\]
where $j$ is the first index such that $(f^{-1}(\aa), \Closure<\cM_0>(f^{-1}(\aa)), j) \in \compcK[\cM_0]$. 
Clearly $(\alpha_0, \alpha_1)$ is the desired isomorphism of computable ages. 
\end{proof}

The next definition will be important for constructing an infinite structure given an age.

\begin{definition}
If $\cM$ is a computable structure we write $\EmbedInfo(\cM)$ to mean $\EmbedInfo(\compcK[\cM])$.  
\end{definition}

The following result will allow us to convert a computable sequence of compatible embeddings into a structure. 
Its proof is immediate from well-known results (see, e.g., \cite[Lemma 2.9]{Computable-Fraisse}), but we include a proof for completeness.
\begin{restatable}{proposition}{CompLimitSequencesStructures}
\label{Computable limit of sequences of structure}
Let $\compcK$ be a computable age
and suppose 
$(F_i)_{i \in \w}$
is a computable sequence of embeddings where for each $i \in \w$ there is a $k_i \in \w$ such that $\dom(F_i) = \AgeIndex[\compcK](k_i)$ and $\codom(F_i) =\AgeIndex[\compcK](k_{i+1})$.

Then there is a computable age $\compcK^*$ which is computably isomorphic to $\compcK$, and a computable
sequence $(\AgeIndex[\compcK](\ell_i))_{i \in\w}$ 
such that for $i \in \w$, we have
%\begin{itemize}
%\item 
$\AgeTuple[\compcK^*](\ell_i) \clsim \AgeTuple[\compcK](k_i)$ and
$\AgeStr[\compcK^*](\ell_i) \subseteq \AgeStr[\compcK^*](\ell_{i+1})$, and
%
%\item
%\begin{eqnarray*}
\[
\ClosureMap[\AgeTuple[\compcK](k_{i+1})](\AgeTuple[\compcK^*](\ell_{i+1})) \circ \ClosureMap[\AgeTuple[\compcK](k_i)](\range(F_i))
= \ClosureMap[\AgeTuple[\compcK](k_i)](\AgeTuple[\compcK^*](\ell_i)).
\]
%\hspace*{90pt} 
%\end{eqnarray*}
%\end{itemize}  
%\end{proposition} 
\end{restatable}

\input{ingredients/CompLimitSequencesStructuresProof.tex}

\begin{corollary}
\label{cor:union-of-comp-chain}
Let $\compcK$ be a computable age
and suppose 
$(F_i)_{i \in \w}$ are as in \cref{Computable limit of sequences of structure}.
Then there is a structure 
$\cD_\w$ 
along with embeddings $G_i\:\AgeStr[\compcK](k_i) \to \cD_\w$ (for $i\in\w$) such that for any $i < j$, we have $G_i = G_j \circ F_{j-1} \circ F_{j-2} \circ \dots  \circ F_i$. 
Further, both $\cD_\w$ and $(G_n)_{n \in \w}$ are computable. 
\end{corollary}
\begin{proof}
Apply \cref{Computable limit of sequences of structure}, and 
let $\cD_\w = \bigcup_{n \in \w} \AgeStr[\compcK^*](\ell_n)$.
Then the following diagram commutes, where $G_i$ is 
$
\ClosureMap[\AgeTuple(k_i)](\AgeTuple[\compcK^*](\ell_i))$
for $i\in\w$.
\begin{center}
\begin{tikzcd}
\AgeStr[\compcK](k_0) \arrow[r, "F_0"] \arrow[d,"G_0" ,""']
& \AgeStr[\compcK](k_1) \arrow[d, "G_1" ] \arrow[r, "F_1"] & \dots \\
\AgeStr[\compcK^*](\ell_0) \arrow[r, "\subseteq"] & \AgeStr[\compcK^*](\ell_1) \arrow[r, "\subseteq"] & \dots 
\end{tikzcd}
\end{center}
Observe that $\cD_\w$ is the union of a computable sequence of computable structures, hence is computable.
\end{proof}

%%%%  %%%%  %%%%  %%%%  %%%%  %%%%
\subsection{Properties of Computable Ages}
%%%%  %%%%  %%%%  %%%%  %%%%  %%%%

We now define computable analogues of various properties of ages. 

\begin{definition}
\label{Definition of (HP)}
An age $\cK$ has the \defn{hereditary property}, written $\HP$, if for all $\cB \in \cK$, we have
$\cA \subseteq \cB$ implies $\cA \in \cK$. 

Suppose that $\StandardTuringDegree$ is a Turing degree.
We say a computable age $\compcK$  has the \defn{$\StandardTuringDegree$-computable hereditary property}, written $\CHP(\StandardTuringDegree)$,
if there is an $\StandardTuringDegree$-computable function that takes a pair $(\AgeIndex(i), \bb)$ where $\bb \in \AgeStr(i)$ and returns some $\AgeIndex(j)$ such that $\AgeTuple(j) \clsim \bb$.
When $\StandardTuringDegree$ is $\TuringDegreeZero$, we speak of the \defn{computable hereditary property}, and write $\CHP$. 
\end{definition}

We have the following bound. 
\begin{lemma}
\label{HP implies 0'-computable HP}
Suppose $\compcK$ is a computable age with $\HP$. Then $\compcK$ has $\CHP(\TuringDegree[\EmbedInfo(\compcK)])$. 
\end{lemma}
\begin{proof}
Given $(\AgeIndex(i), \bb)$ with $\bb \in (\AgeStr(i), \AgeTuple(i))$, use $\EmbedInfo(\compcK)$ to search for $\AgeIndex(j)$ such that $\AgeTuple(j) \clsim \bb$. 
\end{proof}

Not every computable age with $\HP$ has $\CHP$, but every computable age with $\HP$ is isomorphic to one with $\CHP$,
as shown in
\cite{Computable-Fraisse};
we include a proof for completeness.

\begin{proposition}
[{{\cite[Theorem~2.8]{Computable-Fraisse}}}]
\label{Computable age with (HP) has isomorphic copy with (CHP)}
Suppose $\compcK$ is a computable age with $\HP$. Then there is a computable age $\compcK^*$ with $\CHP$ that is isomorphic to $\compcK$. 
\end{proposition}
\begin{proof}
We obtain $\compcK^*$ from $\compcK$ by enumerating all the structures generated by finite tuples in structures in $\compcK$. Because $\compcK$ has $\HP$, it is clear that $\compcK$ and $\compcK^*$ represent the same age and hence are isomorphic.
\end{proof}

\begin{definition}
\label{Definition of (JEP)}
An age $\cK$ has the \defn{joint embedding property}, written $\JEP$, if for all $\cA, \cB \in \cK$ there is a $\cC \in \cK$ for which there are embeddings $\alpha_{\cA}\:\cA \to \cC$ and $\alpha_{\cB}\: \cB \to \cC$. 

Suppose $\StandardTuringDegree$ is a Turing degree.
We say a computable age $\compcK$  has the \defn{$\StandardTuringDegree$-computable joint embedding property}, written $\CJEP(\StandardTuringDegree)$, if there is an $\StandardTuringDegree$-computable function $f$ with domain $\compcK \times \compcK$ such that
for all $\AgeIndex(\ell_0), \AgeIndex(\ell_1) \in \compcK$, we have
$f(\AgeIndex(\ell_0), \AgeIndex(\ell_1))= (F_0, F_1)$ where
%\begin{itemize}
%\item 
(i)
$F_0$ and $F_1$ are embeddings; 
(ii) $\dom(F_0) = \AgeIndex(\ell_0)$
and $\dom(F_1) = \AgeIndex(\ell_1)$; and
(iii) $\codom(F_0) = \codom(F_1)$.
%\end{itemize}
%
When $\StandardTuringDegree$ is $\TuringDegreeZero$, we speak of the \defn{computable joint embedding property}, and write $\CJEP$. 
\end{definition}

\begin{proposition}
\label{(JEP) implies 0' computable (JEP)}
Suppose $\compcK$ is a computable age with $\JEP$. Then $\compcK$ also has $\CJEP(\TuringDegree[\EmbedInfo(\compcK)])$.
\end{proposition}
\begin{proof}
Given $\AgeIndex(i)$ and $\AgeIndex(j)$, search through 
structures $\AgeStr(k) \in \compcK$ 
and finite tuples in $\AgeStr(k)$ 
until one finds an $\ell$ and tuples $\bb_i, \bb_j \in \AgeStr(\ell)$ such that $\bb_i \clsim \AgeTuple(i)$ and $\bb_j \clsim \AgeTuple(j)$. 
\end{proof}

The following two results are due to \cite{Computable-Fraisse}.

\begin{proposition}[{{\cite[Lemma~2.6]{Computable-Fraisse}}}]
\label{Canonical age of computable structure has (CHP) and (CJEP)}
If $\cM$ is a computable structure then $\compcK[\cM]$ has $\CHP$ and $\CJEP$. 
\end{proposition} 

\begin{proposition}[{{\cite[Theorem~2.10]{Computable-Fraisse}}}]
\label{other direction of Canonical age of computable structure has (CHP) and (CJEP)}
If $\compcK$ is a computable age with $\CHP$ and $\CJEP$, then it is canonically the age of some computable structure. 
\end{proposition}

Next we introduce some technical notions 
relevant to
amalgamation properties.  Key among these is a \emph{span}, which serves as the ``base'' of an amalgamation diagram.

\begin{definition}
A \defn{potential span} in a computable age $\compcK$ is a pair $(F_0, F_1)$ of potential embeddings with $\dom(F_0) = \dom(F_1)$.
We say that a potential span $(F_0, F_1)$ is \defn{over} $\dom(F_0)$. 
A potential span is a \defn{span} if in addition, $F_0$ and $F_1$ are both embeddings.

An \defn{amalgamation diagram} over a potential span $(F_0, F_1)$ is a tuple $(G_0, G_1)$ such that 
\begin{itemize}

\item $G_0$ is an embedding and $G_1$ is a potential embedding, 

\item $\dom(G_i) = \codom(F_i)$ for $i \in \{0, 1\}$, 

\item $\codom(G_0) = \codom(G_1)$, and

\item if $(F_0, F_1)$ is a span then 
%\begin{itemize}
$G_1$ is an embedding and
$G_0 \circ F_0 = G_1 \circ F_1$.
%\end{itemize}
\end{itemize}
\end{definition}

When $(F_0, F_1)$ is a span,
the following diagram must commute.
\begin{center}
\begin{tikzcd}
\AgeIndex(i_1) \arrow[dashed, "G_0"]{r} & \AgeIndex(k)\\
\AgeIndex(j) \arrow[r, "F_0"]  \arrow[u, "F_1"] & \AgeIndex(i_0) \arrow[dashed, "G_1"]{u}
\end{tikzcd}
\end{center}

Note that it is computable to check whether or not a tuple is a potential span, but it need not be computable to check whether or not a tuple is a span. Similarly, it need not be computable to check whether a tuple is an amalgamation diagram over a potential span.

\begin{definition}  
\label{Definition of (AP)}
An age $\cK$ has the \defn{amalgamation property}, written $\AP$, if whenever $\cA, \cB, \cC \in \cK$ and whenever $f_{\cB}\:\cA \to \cB$ and $f_{\cC}\:\cA \to \cC$ are embeddings, there is a $\cD \in \cK$ such that there are embeddings $\alpha_{\cB}\:\cB \to \cD$ and $\alpha_{\cC}\:\cC \to \cD$ where $\alpha_{\cB} \circ f_{\cB} = \alpha_{\cC} \circ f_{\cC}$.

Suppose $\StandardTuringDegree$ is a Turing degree.
We say a computable age $\compcK$  has the \defn{$\StandardTuringDegree$-computable amalgamation property}, written $\CAP(\StandardTuringDegree)$, if there is an $\StandardTuringDegree$-computable function which maps each potential span to an amalgamation diagram over it. 
\end{definition}

\begin{restatable}{proposition}{APimpliesJump}
\label{(AP) implies 0' computable (AP)}
Suppose $\compcK$ is a computable age with $\AP$. Then $\compcK$ has 
$\CAP(\TuringDegree[\EmbedInfo(\compcK)])$.  
\end{restatable}

\input{ingredients/APimpliesJumpProof.tex}

The focus of this paper is a weakening of the amalgamation property, known as the cofinal amalgamation property; we will see that this is
computationally more complicated than the amalgamation property.

\begin{definition}
\label{Collection of distinguished extensions}
Suppose $\compcK$ is a computable age. A \defn{collection of distinguished extensions}, written $\DisExt(\compcK)$, is a collection of embeddings such that 
\begin{itemize}

\item for $F \in \DisExt(\compcK)$ we have $\dom(F), \codom(F) \in \compcK$, 

\item for $F \in \DisExt(\compcK)$ and isomorphisms $G$ with $\codom(G) = \dom(F)$ and $\dom(G) \in \compcK$ we have $F \circ G \in \DisExt(\compcK)$, 

\item for $F \in \DisExt(\compcK)$ and isomorphisms $H$ with $\dom(H) = \codom(F)$ and $\codom(H) \in \compcK$ we have $H \circ F \in \DisExt(\compcK)$, 

\item for all $\AgeIndex(i) \in \compcK$ there is an $F \in  \DisExt(\compcK)$ with $\dom(F) = \AgeIndex(i)$.  
\end{itemize}
Such a collection $\DisExt(\compcK)$ is said to be $\StandardTuringDegree$-computable when it is $\StandardTuringDegree$-computable as a set. 
\end{definition}

\begin{definition}
Suppose $\bm{\alpha} = (\alpha_0, \alpha_1)\:\compcK_0 \rightarrow \compcK_1$ is a computable isomorphism and $\DisExt(\compcK_0)$ is a collection of distinguished extensions in $\compcK_0$.  Let $\bm{\alpha}(\DisExt(\compcK_0))$ be the collection of embeddings $F$ in $\CAT(\compcK_1)$ such that $\bm{\alpha}^{-1}(F) \in \DisExt(\compcK_0)$. 
\end{definition}

\begin{lemma}
\label{Collections of distinguished extensions are closed under computable isomorphisms}
Suppose $\bm{\alpha} = (\alpha_0, \alpha_1)\:\compcK_0 \rightarrow \compcK_1$ is a computable isomorphism and $\DisExt(\compcK_0)$ is a computable collection of distinguished extensions in $\compcK_0$. Then $\bm{\alpha}(\DisExt(\compcK_0))$ is a computable collection of distinguished extensions in $\compcK_1$. 
\end{lemma}
\begin{proof}
This holds as 
$\DisExt(\compcK_0)$ is closed under  
pre- and post-compositions of isomorphisms.
\end{proof}

\begin{definition}
\label{Definition of (coAP)}
An age $\cK$ has the \defn{cofinal amalgamation property}, written $\coAP$, if for every $\cA \in \cK$ there is an $\cA' \in \cK$, called an \defn{amalgamation base},  such that 
\begin{itemize}
\item there is an embedding $\beta\:\cA \to \cA'$ and

\item whenever $\cB, \cC \in \cK$ and $f_{\cB}\:\cA' \to \cB$ and $f_{\cC}\:\cA'\to \cC$ are embeddings, there is a $\cD \in \cK$ and there are embeddings $\alpha_{\cB}\:\cB \to \cD$ and $\alpha_{\cC}\:\cC \to \cD$ where $\alpha_{\cB} \circ f_{\cB} = \alpha_{\cC} \circ f_{\cC}$. 
\end{itemize} 
\end{definition}

\begin{definition}
\label{Definition of (CcoAP)}
Suppose $\compcK$ is a computable age. We define a \defn{witness to $\coAP$} to be a pair $(\WitcoAP(\compcK), \WitFuncoAP(\compcK))$ satisfying the following. 

\begin{itemize}
\item[(a)] $\WitcoAP(\compcK)$ is a collection of distinguished extensions such that if $F \in \WitcoAP(\compcK)$,  then $\codom(F)$ is an amalgamation base, and

\item[(b)] $\WitFuncoAP(\compcK)$ takes as input tuples $(F, G_0, G_1)$ such that 
%\begin{itemize}
%\item 
(i)
$F \in \WitcoAP(\compcK)$, 
%
%\item 
(ii)
$(G_0, G_1)$ is a potential span, and
%
%\item 
(iii)
$\dom(G_0) = \dom(G_1) = \codom(F)$, 
%\end{itemize}
and outputs an amalgamation diagram over $(G_0, G_1)$. 
\end{itemize}
Such a witness is said to be \defn{$\StandardTuringDegree$-computable} if $\WitcoAP(\compcK)$ is $\StandardTuringDegree$-c.e.\ and $\WitFuncoAP(\compcK)$ is an $\StandardTuringDegree$-computable partial function. 
We say that $\compcK$ has $\CcoAP(\StandardTuringDegree)$ if it has an $\StandardTuringDegree$-computable witness to $\coAP$.
\end{definition}

\begin{definition}   
Let $\bm{\alpha} = (\alpha_0, \alpha_1)\:\compcK_0 \rightarrow \compcK_1$ and suppose $(\WitcoAP(\compcK_0), \WitFuncoAP(\compcK_0))$ is a witness to $\coAP$ for $\compcK_0$. Let $\bm{\alpha}(\WitFuncoAP(\compcK_0))$ be the map which takes as input those tuples $(F, G_0, G_1)$  such that  
%\begin{itemize}
%\item 
(i)
$F \in \bm{\alpha}(\WitcoAP(\compcK_0))$, 
(ii)
%\item 
$(G_0, G_1)$ is a potential span in $\compcK_1$,  and
(iii)
%\item 
$\dom(G_0) = \dom(G_1) = \codom(F)$, 
%\end{itemize}
and outputs $(\bm{\alpha}(H_0), \bm{\alpha}(H_1))$ where
$(H_0, H_1) = \WitFuncoAP(\compcK_0)\bigl(\bm{\alpha}^{-1}(F),\, \bm{\alpha}^{-1}(G_0),\, \bm{\alpha}^{-1}(G_1)\bigr).$
\end{definition}

\begin{lemma}
\label{Inverse image of cofinal witness under computable isomorphism is cofinal witness}
Let $\bm{\alpha}=(\alpha_0, \alpha_1)\:\compcK_0 \to \compcK_1$ be a computable isomorphism, and suppose that $(\WitcoAP(\compcK_0), \WitFuncoAP(\compcK_0))$ is a witness to $\coAP$ for $\compcK_0$ which is computable. Then $(\bm{\alpha}(\WitcoAP(\compcK_0)), \bm{\alpha}(\WitFuncoAP(\compcK_0)))$ is a witness to $\coAP$ for $\compcK_1$ which is computable. 
\end{lemma}

\begin{proof}
By \cref{Collections of distinguished extensions are closed under computable isomorphisms},
the set $\bm{\alpha}(\WitcoAP(\compcK_0))$ is a collection of distinguished extensions. 
\cref{Definition of (CcoAP)} (a) holds, as being an amalgamation base is closed under isomorphism of computable ages. 
\cref{Definition of (CcoAP)} (b) holds,  
as being an amalgamation of a potential span is closed under isomorphisms of computable ages.
Finally, $(\bm{\alpha}(\WitcoAP(\compcK_0)), \bm{\alpha}(\WitFuncoAP(\compcK_0)))$ is computable because it can be computed from $(\WitcoAP(\compcK_0), \WitFuncoAP(\compcK_0))$ and $\bm{\alpha}$, both of which are computable. 
\end{proof}

The following corollary is immediate. 

\begin{corollary}
Let $\StandardTuringDegree$ be a Turing degree, and suppose $\compcK_0$ has 
$\CcoAP(\StandardTuringDegree)$ and $\compcK_1$ is computably isomorphic to $\compcK_0$. Then $\compcK_1$ has $\CcoAP(\StandardTuringDegree)$.
\end{corollary}

\begin{proposition}
The following are equivalent for a computable age $\compcK$. 
\begin{itemize}
\item $\compcK$ has $\coAP$. 

\item  $\compcK$ has $\CcoAP(\StandardTuringDegree)$ for some Turing degree $\StandardTuringDegree$.
\end{itemize}
\end{proposition}
\begin{proof}
Note that any witness to $\coAP$ can be encoded by a single real. The computable age $\compcK$ has $\coAP$ if and only if there is some witness to $\coAP$ for $\compcK$, which then has $\CcoAP(\StandardTuringDegree)$ where $\StandardTuringDegree$ is the Turing degree of the real encoding the witness.
\end{proof}

\begin{lemma}
If an age 
has $\AP$, then it also has $\coAP$. 
If a computable age 
has $\CAP$, then it also has $\CcoAP$.
\end{lemma}
\begin{proof}
This follows since in an age with $\AP$, every element is an amalgamation base.
\end{proof}

Earlier we saw how to obtain
bounds on the computational power needed for the effective versions of 
each of $\HP$, $\JEP$, and $\AP$.
We now do analogously with
$\coAP$, but the bounds are higher.
Our next result is technical but important.
\begin{restatable}{theorem}{UpperBoundCompCAP}
\label{Upper bounds on the computability of (CAP)}
Let $\compcK$ be a computable age, and let
$\WitcoAP$ be the collection of all embeddings between elements of $\compcK$ whose codomain is some amalgamation base in $\compcK$.  
Then (a) $\TuringDegree[\WitcoAP] \Turingleq \TuringDoubleJump[\EmbedInfo(\compcK)]$,
and (b)~when $\compcK$ has $\coAP$, it has $\CcoAP(\TuringDoubleJump[\EmbedInfo(\compcK)])$.
\end{restatable}

\input{ingredients/UpperBoundCompCAPProof.tex}

%%%%%%%%%%%%%%%%%%%%%%%%%%%%%%%%%%%%%%%%

%%%%  %%%%  %%%%  %%%% %%%%  %%%%  %%%%  %%%%
\section{Upper Bounds}
\label{sec:upper}
%%%%  %%%%  %%%%  %%%% %%%%  %%%%  %%%%  %%%%

%%%%  %%%%  %%%%  %%%%  %%%%  %%%%
\subsection{Ultrahomogeneous Structures} 
%%%%  %%%%  %%%%  %%%%  %%%%  %%%%

\begin{definition}
A structure $\cM$ is \defn{ultrahomogeneous} if for every finitely generated structure $\cA \subseteq \cM$ and every isomorphism $f\:\cA \to \cA^*$ with $\cA^* \subseteq \cM$ there is an automorphism $g\:\cM \to \cM$ such that $g \rest[\cA] = f$. 

The structure $\cM$ is a \defn{\Fraisse\ limit of $\cK$} if the age of $\cM$ is $\cK$ and $\cM$ is ultrahomogeneous. 
\end{definition}

\begin{definition}
A computable structure $\cM$ is \defn{computably ultrahomogeneous} if there is a function $g$ which takes in pairs of finite tuples $(\aa, \bb)$ of equal length from $\cM$ and returns a code for some bijection of the underlying set of $\cM$ with itself, such that 
if $\ClosureMap[\aa](\bb)$ is an isomorphism from $\Closure<\cM>(\aa)$ to $\Closure<\cM>(\bb)$,
then $g(\aa, \bb)$ is an automorphism of $\cM$ extending $\ClosureMap[\aa](\bb)$. We call such a map $g$ a \defn{witness} to the computable ultrahomogeneity of $\cM$.

The structure $\cM$ is a \defn{computable \Fraisse\ limit of $\compcK$} if $\cM$ is computably ultrahomogeneous and $\compcK$ is a computable representation of the age of $\cM$.  
\end{definition}

Note that the paper \cite{Computable-Fraisse} defines a notion of ``computable homogeneity'' where instead of taking in two $\clsim$-equivalent tuples and returning an automorphism it takes in two $\clsim$-equivalent tuples $\aa, \bb$ and an extra point $x$ and returns a new point $y$ such that $\aa y \clsim \bb x$. Note that their notion of computable homogeneity is equivalent to our notion of computable ultrahomogeneity via a standard back-and-forth construction of an automorphism from one-point extension axioms.

Our next result
strengthens
one direction of \cite[Theorem~3.9]{Computable-Fraisse}.

\begin{restatable}{proposition}{cCompUltracCap}
\label{c-computably ultrahomogeneity implies c-CAP}
Suppose $\cM$ is an arbitrary
countable structure with $\StandardTuringDegree$-computable domain which is $\StandardTuringDegree$-computably ultrahomogeneous. Then $\compcK[\cM]$ has $\CAP(\StandardTuringDegree)$.
\end{restatable}

\input{ingredients/cCompUltracCapProof.tex}

The other direction of \cite[Theorem~3.9]{Computable-Fraisse} is equivalent to the following result.

\begin{proposition}
[{{\cite[Theorem~3.9]{Computable-Fraisse}}}]
\label{Computable Fraisse limit from (HP) (JEP) (AP)}
Suppose  $\compcK$ is a computable age that has $\CHP$, $\CJEP$, and $\CAP$. 
Then $\compcK$ has a computable \Fraisse\ limit.
\end{proposition} 

We have the following
corollary of \cref{HP implies 0'-computable HP} and \cref{(JEP) implies 0' computable (JEP),(AP) implies 0' computable (AP),Computable Fraisse limit from (HP) (JEP) (AP)}.

\begin{corollary}
\label{Bound on the computability of Fraisse limit}
Suppose $\compcK$ is a computable age with $\HP$, $\JEP$, and $\AP$. Then $\compcK$ has an $\EmbedInfo(\compcK)$-computable \Fraisse\ limit. 
\end{corollary}

The next result uses a standard back-and-forth construction to build the appropriate automorphism.

%\begin{proposition}
\begin{restatable}{proposition}{UltraHomogeneousBound}
\label{Ultrahomogeneous implies EmbedInfo-computably ultrahomogeneous}
Let $\cM$ be a computable structure, and 
suppose $\cM$ is ultrahomogeneous. Then $\cM$ is $\EmbedInfo(\cM)$-computably ultrahomogeneous. 
%\end{proposition}
\end{restatable}

\input{ingredients/UltraHomogeneousBoundProof.tex}

%%%%  %%%%  %%%%  %%%%  %%%%  %%%%
\subsection{Cofinally Ultrahomogeneous Structures}
%%%%  %%%%  %%%%  %%%%  %%%%  %%%%

\begin{definition}     
Let $\cM$ be a structure. A set $E$ is a \defn{cofinal collection} (in $\cM$) when the following hold.
\begin{itemize}

\item $E$ is a collection of finite tuples in $\cM$. 

\item For every finite tuple $\aa$ in $\cM$
there is some $\bb \in E$ such that $\aa \subseteq \Closure<\cM>(\bb)$.

\item If 
$\bb \in E$ and if
$\aa \subseteq \Closure<\cM>(\bb)$ is a finite tuple,
then $\aa \bb \in E$.

\item For all finite tuples $\aa$ and $\bb$ in $\cM$, if $\aa \in E$ and $\aa \clsim \bb$, then $\bb \in E$.
\end{itemize}

\end{definition}

\begin{definition}
\label{Definition of cofinally ultrahomogeneous}
A structure $\cM$ is \defn{cofinally ultrahomogeneous} if for every finite set $A \subseteq \cM$ there 
is a
finitely generated $\cB \subseteq \cM$ with $A \subseteq \cB$ such that whenever $f\:\cB \to \cB^*$ is an isomorphism with $\cB^* \subseteq \cM$ then there is an automorphism $f^*\:\cM \to \cM$ such that $f^* \rest[\cB] = f$. 
Given a cofinal collection $E$ in $\cM$,
we say that $\cM$ is \defn{cofinally ultrahomogeneous for} $E$ if
every isomorphism between two elements of $E$ can be extended to an automorphism of $\cM$.
\end{definition}
(In particular, for every cofinally ultrahomogeneous structure $\cM$, there is some cofinal collection $E$ in $\cM$ such that $\cM$ is cofinally ultrahomogeneous for $\cM$.)

The following is an example from \cite{deRancourtSlides} of a cofinally ultrahomogeneous structure which is not ultrahomogeneous. For more details on this example as an instance of cofinal amalgamation, see \cite[Example 3.1.9]{MR3697592}.

\begin{example}
Let $Z = (G, E)$ be the unique (up to isomorphism) infinite undirected irreflexive graph with a single connected component such that every element has degree $2$, i.e., $Z$ consists of a single $\Integers$-chain. For distinct $a, b, c \in Z$ we say that $c$ is \defn{between} $a$ and $b$ if there is a $c_0, \dots, c_{n-1}$ such that $(a, c_0)$, $(c_0, c_1), \dots, (c_{n-1}, b) \in E$ and $c \in \{c_i\}_{i \in [n-1]}$. We also say that $a$ and $b$ are distance $n+1$ apart if the collection of elements between $a$ and $b$ has size $n$.

Note that $Z$ is not ultrahomogeneous, as whenever $a_0, a_1, b_0, b_1\in Z$ are such that neither pair $(a_0, a_1)$ nor $(b_0, b_1)$ has an edge, then these pairs have the same quantifier-free type, even when the distance between $a_0$ and $a_1$ is different than the distance between $b_0$ and $b_1$. 

Call a subset $A$ of $Z$ \defn{closed} if whenever $a, b \in A$ are distinct and $c$ is between $a$ and $b$ then $c \in A$. Note that given any two finite closed sets of the same size, there is an automorphism of $Z$ taking one to the other. Also note that every finite set of elements is contained in a finite closed set. Therefore $Z$ is cofinally ultrahomogeneous. 
\lipicsEnd
\end{example}

\begin{definition}
Let $\cM$ be a computable structure and 
suppose $E$ is an $\StandardTuringDegree$-computable cofinal collection in $\cM$.
We say that
$\cM$ is \defn{$\StandardTuringDegree$-computably cofinally ultrahomogeneous} (with respect to $E$) if there is an $\StandardTuringDegree$-computable $g$ such that 
\begin{itemize}

\item $\dom(g)$ is the collection of pairs $(\aa, \bb)$ of tuples in $\cM$ where $\aa \in E$ and $\aa \tuplesim \bb$, 

\item if $(\aa, \bb)\in \dom(g)$, then $g(\aa, \bb)$ is the index of an $\StandardTuringDegree$-computable bijection $g_{\aa, \bb}\: \cM \to \cM$, where 
%\begin{itemize}
%\item 
(i) $g_{\aa, \bb}(\aa) = \bb$; and
(ii) if $\aa \clsim \bb$ then $g_{\aa, \bb}$ is an automorphism of $\cM$. 
%\end{itemize}
\end{itemize}
\end{definition}

\begin{restatable}{proposition}{cofIffScompForSomeS}
Let $\cM$ be a computable structure.
The following are equivalent.
\begin{itemize}
\item[(a)] $\cM$ is cofinally ultrahomogeneous.

\item[(b)] There is some Turing degree $\StandardTuringDegree$ such that $\cM$ is $\StandardTuringDegree$-computably cofinally ultrahomogeneous. 

\end{itemize}
\end{restatable}

\input{ingredients/cofIffScompForSomeSProof.tex}

\begin{restatable}{proposition}{cofUltraImpliescofUltraAmalgamationBases}
\label{Cofinal ultrahomogeneous implies cofinal ultrahomogeneous with respect to amalgamation bases}
Let $\cM$ be a cofinally ultrahomogeneous structure for 
some 
cofinal collection $E$. If $\aa \in E$, then 
$\Closure<\cM>(\aa)$ is an amalgamation base in $\compcK[\cM]$.
\end{restatable}

\input{ingredients/cofUltraImpliescofUltraAmalgamationBasesProof.tex}

\begin{restatable}{proposition}{cofUltraIsoExtendsAutomorphism}
\label{Cofinal ultrahomogeneous implies isomorphism of amalgamation bases extend to automorphisms}
Let $\cM$ be a cofinally ultrahomogeneous structure.
Let $\cA_0, \cA_1 \subseteq \cM$ be amalgamation bases in the age of $\cM$ and suppose $k\:\cA_0 \to \cA_1$ is an isomorphism. 
Then there is an automorphism $g\:\cM \to \cM$ such that $g\rest[\cA_0] = k$. 
\end{restatable}

\input{ingredients/cofUltraIsoExtendsAutomorphismProof.tex}

\begin{restatable}{proposition}{cCompCofImpcCcoAP}
\label{c-computably cofinal ultrahomogeneity implies c-CcoAP}
Let $\StandardTuringDegree$ be a Turing degree, and let
$\cM$ be an $\StandardTuringDegree$-computable structure that
is $\StandardTuringDegree$-computably cofinally ultrahomogeneous. 
If $\TuringDegree[\EmbedInfo(\cM)] \Turingleq \StandardTuringDegree$ then 
$\compcK[\cM]$ has 
$\CcoAP(\StandardTuringDegree)$.
\end{restatable}

\input{ingredients/cCompCofImpcCcoAPProof.tex}

We already knew, by
\cref{Canonical age of computable structure has (CHP) and (CJEP)}, 
that for any $\StandardTuringDegree$-computable $\StandardTuringDegree$-cofinally ultra\-homogeneous structure $\cM$, its canonical $\StandardTuringDegree$-computable age $\compcK[\cM]$ has $\CHP(\StandardTuringDegree)$ and $\CJEP(\StandardTuringDegree)$. 
\cref{c-computably cofinal ultrahomogeneity implies c-CcoAP}
shows that 
it also has
$\CcoAP(\StandardTuringDegree)$
if $\TuringDegree[\EmbedInfo(\cM)] \Turingleq \StandardTuringDegree$.

\begin{definition} 
\label{Definition of cofinal extension property}
Let $\StandardTuringDegree$ be a Turing degree, and suppose $E$ is a cofinal collection. An $\StandardTuringDegree$-computable structure $\cM$ has the \defn{$\StandardTuringDegree$-computable cofinal extension property} (with respect to $E$) if there is an $\StandardTuringDegree$-computable partial function $p$ which takes tuples $(\aa, \bb, \cc)$ with
(i)~$\aa, \bb,\cc \in \cM$, 
(ii)~$\aa, \cc  \in E$, 
(iii)~$\aa \subseteq \Closure<\cM>(\cc)$, and
(iv)~$\aa \tuplesim \bb$, 
%\end{itemize}
and which returns a tuple $\dd$ where $\aa\cc \tuplesim \bb\dd$ and whenever $\aa \clsim \bb$,  
then 
%\begin{itemize}
%\item 
(I)~$\dd \clsim \cc$, 
%\item 
(II)~$\bb \in \Closure<\cM>(\dd)$, and
%\item 
(III)~$\ClosureMap[\dd](\cc) \rest[\bb] = \ClosureMap[\bb](\aa)$.
\end{definition}

The next result describes an important relationship between 
the previous two definitions.
The proof uses the
cofinal extension property to build an automorphism extending any isomorphism of elements in $E$, generalizing the proof that a structure is ultrahomogeneous if and only if its age has the 1-point extension property. 

\begin{restatable}{proposition}{EquivBetweenCofinal}
\label{Equivalence between cofinal ultrahomogeneity and cofinal extension property}
Let $\StandardTuringDegree$ be a Turing degree, and
let $\cM$ be an $\StandardTuringDegree$-computable structure.
Suppose that $\EmbedInfo(\cM) \Turingleq \StandardTuringDegree$ and let
$E$ be an $\StandardTuringDegree$-computable collection of cofinal pairs. 
Then $\cM$ is $\StandardTuringDegree$-computably cofinally ultrahomogeneous with respect to $E$ if and only if
$\cM$ has the $\StandardTuringDegree$-computable cofinal extension property with respect to $E$.
\end{restatable}

\input{ingredients/EquivBetweenCofinalProof.tex}

\begin{definition}
Suppose $\cK$ is an age in $\Lang$. A countable $\Lang$-structure $\cM$ is a \defn{cofinal \Fraisse\ limit} of $\cK$ if 
$\cM$ is cofinally ultrahomogeneous and
the age of $\cM$ is $\cK$. 
\end{definition}

\begin{definition}
Suppose $\compcK$ is a computable age. An $\StandardTuringDegree$-computable structure $\cM$ is an $\StandardTuringDegree$-computable cofinal \Fraisse\ limit if $\cM$ is $\StandardTuringDegree$-computably cofinally ultrahomogeneous and $\compcK$ is an $\StandardTuringDegree$-computable representation of the age of $\cM$. 
\end{definition}

We prove the next result
under a mild technical assumption, which we describe in the first paragraphs of the proof. (It is straightforward to generalize this to a proof of the full result, but this requires additional bookkeeping using 
\cref{Computable limit of sequences of structure}.)

\begin{restatable}{theorem}{compCofFrLimit}
\label{prop:computable-cofinal-Fr-limit}
Let $\StandardTuringDegree$ be a Turing degree.
Suppose  that
%\begin{itemize}
%
%\item[(a)] 
$\compcK$ is an $\StandardTuringDegree$-computable age,  
%
%\item[(b)] 
$\TuringDegree[\EmbedInfo(\compcK)] \Turingleq \StandardTuringDegree$, and
%
%\item[(c)] 
$\compcK$ has $\CHP(\StandardTuringDegree)$, $\CJEP(\StandardTuringDegree)$, and $\CcoAP(\StandardTuringDegree)$. 
%\end{itemize}
Then there is an $\StandardTuringDegree$-computable cofinal \Fraisse\ limit of $\compcK$. 
\end{restatable}

\input{ingredients/compCofFrLimitProof.tex}

The next result follows from
\cref{HP implies 0'-computable HP}, \cref{(JEP) implies 0' computable (JEP)}, 
and \cref{Upper bounds on the computability of (CAP),prop:computable-cofinal-Fr-limit}.

\begin{corollary}
\label{Bound on the computability of cofinal Fraisse limit}
Suppose $\compcK$ is a computable age with $\HP$, $\JEP$, and $\coAP$. Then $\compcK$ has an $\TuringDoubleJump[\EmbedInfo(\compcK)]$-computable cofinal \Fraisse\ limit. 
\end{corollary}

Putting together several of these ingredients, we have the following result.

\begin{restatable}{theorem}{cofUltrDoubleJump}
\label{Cofinal ultrahomogeneous implies DoubleJump(EmbedInfo)-cofinal ultrahomogeneous}
Suppose $\cM$ is a cofinally ultrahomogeneous structure. Then $\cM$ is \linebreak $\TuringDoubleJump[\EmbedInfo(\cM)]$-computably cofinally ultrahomogeneous.
\end{restatable}

\input{ingredients/cofUltrDoubleJumpProof.tex}

Combining \cref{Embdedding info is computable from 0'} and 
\cref{Cofinal ultrahomogeneous implies DoubleJump(EmbedInfo)-cofinal ultrahomogeneous}, we obtain the following.

\begin{corollary}
\label{triplejump-corollary}
Suppose $\cM$ is a computable cofinally ultrahomogeneous structure. Then $\cM$ is $\TuringTripleJump$-computably cofinally ultrahomogeneous.
\end{corollary}

%%%%  %%%%  %%%%  %%%%  %%%%  %%%%
\section{Lower Bounds}
\label{sec:LowerBounds}
%%%%  %%%%  %%%%  %%%%  %%%%  %%%%

To conclude the paper, we give 
two main lower bound results,
\cref{Lower bound on coAP in finite language} (for finite relational languages) and
\cref{Lower bound on coAP in infinite language} (for countable languages which may be infinite or include function symbols), and an associated corollary for each in combination with  \cref{c-computably cofinal ultrahomogeneity implies c-CcoAP}.

%%%%  %%%%  %%%%  %%%%  %%%%  %%%%
\subsection{Finite Relational Languages}
%%%%  %%%%  %%%%  %%%%  %%%%  %%%%

We now prove \cref{Lower bound on coAP in finite language}, which constructs
a computable age with $\coAP$ that is the canonical computable age of some structure in a finite relational language, but is such that if it has $\CcoAP(\StandardTuringDegree)$ for some 
$\StandardTuringDegree$, then $\TuringJump \Turingleq \StandardTuringDegree$.

Before the proof, we describe an outline of our approach.
We first
computably divide the infinite structure we are constructing into $\w$-many disjoint pieces such that no relation holds between elements of distinct pieces. 
This 
provides
$\w$-many distinct subages of our canonical age such that our age has $\coAP$ if and only if each 
subage
does. Further, 
as
everything is done computably, we can compute a witness for $\coAP$ in a subage from a witness for $\coAP$ in the full age. 
Hence our task is to build
an age such that from any witness to $\coAP$ we can determine (uniformly in $e$)
whether  or not $\{e\}(0) \converges$. 

We create a structure which can be divided into $\w$-many disjoint pieces by starting with a base language $\Lang_0 = \{C, E\}$. We then let $C$ be the edge relation of a directed graph consisting of infinitely many cycles of every length. We then let $E$ be an equivalence relation such that each equivalence class contains at most one element of each cycle. We will then mandate that no relation holds of a tuple with elements of cycles of different lengths. We will also mandate that no relation holds of a tuple with elements in different equivalence classes and that all relations are preserved if we move each element of the tuple one step along the cycle. In this way the structure is completely determined by the structure on one equivalence class in each cycle.

We now need to build a structure such that from any witness to $\CAP$ for its age we can determine whether or not $\{e\}(0) \converges$. In order to do this, we define a component structure $W_{m, n}$ 
for $m, n \in \w \cup \{\w\}$
such that $m < n$ or $m = n = \w$.
Intuitively $W_{m, n}$ has three parts. First it has two elements $(q_+, q_-)$ which act as the root of the structure. Then connected to $q_+$ we have a $B$-chain of length $m$ and a $Y$-chain of length $n$. Finally connected to $q_-$ we have a $B$-chain of length $n$ and a $Y$-chain of length $m$. The key property of $W_{m, n}$ is that given a substructure $A$ of $W_{m, n}$ containing $\{q_+, q_-\}$ but where the longest $B$-chain and the longest $Y$-chain are both at most $m$, then it is possible to embed $A$ into $W_{m, n}$ in two ways: one in which the $q_+$ gets mapped to an element connected to a $B$-chain of length $n$, and one in which $q_+$ gets mapped to an element connected to a $Y$-chain of length $n$. As these cannot be amalgamated in $W_{m, n}$, this implies that $A$ is not an amalgamation base. Hence if $\{q_+, q_-\} \subseteq A \subseteq W_{m, n}$ and $A$ is an amalgamation base, then $A$ must contain either a $B$-chain or a $Y$-chain of length $> m$, where this chain is attached to either $q_+$ or $q_-$. 

With this component in hand, if $\{e\}(0) \diverges$ we can let the $e$th component of the age simply be $W_{2, 4}$ (plus infinitely many elements which do not interact with it) and if $\{e\}(0) \converges$ we can let the $e$th component of the age be $W_{30, 90}$ (plus infinitely many elements which do not interact with it). Therefore if $A$ is an amalgamation base over $\{q_+, q_-\}$ then $A$ contains either an $B$-chain or $Y$-chain of length $> 4$ connected to one of $q_+$ or $q_-$ if and only if $\{e\}(0) \converges$, as desired.

\begin{restatable}{theorem}{LowerBoundCoAPFinite}
\label{Lower bound on coAP in finite language}
There is a computable age $\compcK$, which is the canonical computable age of some structure in a finite relational language, such that $\compcK$ has $\coAP$ but if it has $\CcoAP(\StandardTuringDegree)$ for some 
$\StandardTuringDegree$, then $\TuringJump \Turingleq \StandardTuringDegree$.
\end{restatable}

\input{ingredients/LowerBoundCoAPFiniteProof.tex}

We then obtain the following corollary
of
\cref{Lower bound on coAP in finite language} and \cref{c-computably cofinal ultrahomogeneity implies c-CcoAP}. 
\begin{corollary}
\label{Lower bound on cofinal Fraisse limit in finite language}
There is a computable age $\compcK$, which is the canonical computable age of some structure in a finite relational language, such that if $\StandardTuringDegree$ is a Turing degree and $\cM$ is 
an $\StandardTuringDegree$-computable structure that is $\StandardTuringDegree$-computably cofinally ultrahomogeneous
and a cofinal \Fraisse\ limit of $\compcK$, then $\TuringJump \Turingleq \StandardTuringDegree$. 
\end{corollary}

%%%%  %%%%  %%%%  %%%%  %%%%  %%%%
\subsection{Arbitrary Countable Languages}
%%%%  %%%%  %%%%  %%%%  %%%%  %%%%

We now turn our attention to \cref{Lower bound on coAP in infinite language}, which constructs a computable age with $\coAP$ that is the canonical computable age of some structure but is such that 
if it has $\CcoAP(\StandardTuringDegree)$ 
for some  $\StandardTuringDegree$,
then $\TuringDoubleJump \Turingleq \StandardTuringDegree$.

Before the proof, we describe an outline of our approach.
As with the proof of \cref{Lower bound on coAP in finite language}, we first divide our structure into $\w$-many  non-interacting parts. However as our language is infinite this can easily be done assuming we have $\w$-many unary relations which partition our structure and that no other relation holds of a tuple that has elements in distinct elements of this partition. Further there is no harm in assuming that our age contains the age constructed in \cref{Lower bound on coAP in finite language}. In particular, we can assume that we can compute $\EmbedInfo(\compcK)$ from any witness to cofinal amalgamation of our age. 

In the construction, we choose a function $f\:\w \times \w \to \{0, 1\}$ such that for any $e \in \w$, $\{e\}^{\TuringJump}(0) \converges$ if and only if $\Phi_e = \{n \st f(e, n) = 0\}$ is finite. With this function in hand we will want to (uniformly in $e$) construct an age such that we can determine from a witness to cofinal amalgamation for the age whether or not $\{n \st f(e, n) = 0\}$ is finite. 

In order to do this we will need to define a given component structure $W_{\sigma}$ for $\sigma\in 2^{<\w}$. This component will consist of a structure whose reduct to the appropriate sublanguage is $W_{|\Phi_e|, 1+|\Phi_e|}$, but which will be in a language that also has $\w$-many unary relation symbols $\{U_i\}_{i \in \w}$. We will then use $\sigma$ to determine whether or not $U_i$ holds for each $i$. Specifically, if $k \not \in \Phi_e$ then $\neg U_i(x)$ will always hold. However, if $k \in \Phi_e$ and $\ell = |\{i < k \st i \in \Phi_e\}|$ then whether or not $U_i(x)$ holds will be determined by the value of $\sigma(\ell)$. 

The effect of this is that if $\len(\sigma) > 1+|\Phi_e|$ then $W_{\sigma\^0}$ and $W_{\sigma\^1}$ are isomorphic. But if $\Phi_e$ is infinite then $W_{\sigma\^0}$ and $W_{\sigma\^1}$ are never isomorphic. Because we can compute $\EmbedInfo(\compcK)$ from any witness to cofinal amalgamation, this reduces the problem of determining whether $\Phi_e$ is infinite to the problem of finding a $g_e$ such that whenever $\Phi_e$ is finite, $g_e \geq 1+|\Phi_e|$. But if $\Phi_e$ is finite and $A \subseteq W_\sigma$ is an amalgamation base containing $\{q_+, q_-\}$ then $A$ must contain either a $B$-chain or a $Y$-chain, which must be of length $1+|\Phi_e|$. So when $\Phi_e$ is finite, we can determine 
$|\Phi_e|$
from any such amalgamation base, as desired.

\begin{restatable}{theorem}{LowerBoundCoAPInfinite}
\label{Lower bound on coAP in infinite language}
There is a computable age $\compcK$, which is the canonical computable age of some structure, such that $\compcK$ has $\coAP$ but if it has $\CcoAP(\StandardTuringDegree)$ 
for some 
$\StandardTuringDegree$,
then $\TuringDoubleJump \Turingleq \StandardTuringDegree$. 
\end{restatable}

\input{ingredients/LowerBoundCoAPInfiniteProof.tex}

We also obtain the following immediate corollary  of \cref{Lower bound on coAP in infinite language} and \cref{c-computably cofinal ultrahomogeneity implies c-CcoAP}.

\begin{corollary}
\label{Lower bound on cofinal Fraisse limit in infinite language}
There is a computable age $\compcK$, which is the canonical computable age of some structure, such that if $\StandardTuringDegree$ is a Turing degree and $\cM$ is 
an $\StandardTuringDegree$-computable structure that is $\StandardTuringDegree$-computably cofinally ultrahomogeneous
and a cofinal \Fraisse\ limit of $\compcK$,
then $\TuringDoubleJump \Turingleq \StandardTuringDegree$. 
\end{corollary}

\bibliographystyle{amsnomr}
\bibliography{bibliography}

\end{document}

%% file: macros.tex
\newtheorem*{corollary*}{Corollary}
\newtheorem*{theorem*}{Theorem}

\def\xx{{\EM{\mbf{x}}}}

\newcommand{\converges}{\!\downarrow}
\newcommand{\diverges}{\!\uparrow}

\newcommand\rest[1][]{\!\upharpoonright_{#1}}

\newcommand\tuplesim{\sim}

\newcommand{\defn}[1]{{\bf{#1}}}

\newcommand{\dashedarrow}[1]
{%
  \EM{\stackrel{#1}{\dashrightarrow}}
}

\DeclareMathOperator{\dom}{dom}
\DeclareMathOperator{\range}{range}

\DeclareMathOperator{\codom}{codom}
\DeclareMathOperator{\Funct}{Funct}
\DeclareMathOperator{\Rel}{Rel}
\DeclareMathOperator{\arity}{ar}

\DeclareDocumentCommand{\forces}{d[]}
{
\IfNoValueTF{#1}
	{
	\EM{\Vdash}
	}
	{
	\EM{\Vdash_{#1}}
	}
}

\DeclareDocumentCommand{\dual}{d()}
{
\IfNoValueTF{#1}
	{
	\EM{\hat{\ }}
	}
	{
	\EM{\hat{#1}}
	}
}

\DeclareDocumentCommand{\SizeAS}{d()}
{
\IfNoValueTF{#1}
	{
	\EM{|\cdot|}
	}
	{
	\EM{|#1|}
	}
}

\DeclareDocumentCommand{\compcK}{d[]}
{
\IfNoValueTF{#1}
	{
	\EM{\mathbb{K}}
	}
	{
	\EM{\mathbb{K}[#1]}
	}
}

\newcommand{\Fraisse}{Fra\"iss\'e}
\newcommand{\FRAISSE}{FRA\"ISS\'E}
\newcommand{\len}{\EM{\mathrm{len}}}

\DeclareDocumentCommand{\halts}{}{\!\downarrow}
\DeclareDocumentCommand{\nohalts}{}{\!\uparrow}

\newcommand{\TuringDegreeZero}{\mathbf{0}}
\newcommand{\StandardTuringDegree}{\TuringDegree[s]}

\DeclareDocumentCommand{\TuringDegree}{d[]}
{
\mathfrak{#1}
}
\DeclareDocumentCommand{\TuringJump}{d[]}
{
\IfNoValueTF{#1}
	{
	\mathbf{0'}
	}
	{
	\TuringDegree[{#1}]'
	}
}

\DeclareDocumentCommand{\TuringDoubleJump}{d[]}
{
\IfNoValueTF{#1}
	{
	\mathbf{0''}
	}
	{
	\TuringDegree[{#1}'']
	}
}
\DeclareDocumentCommand{\TuringTripleJump}{d[]}
{
\IfNoValueTF{#1}
	{
	\EM{\mathbf{0'''}}
	}
	{
	\EM{\mathbf{{#1}'''}}
	}
}

\newcommand{\Turingleq}{\leq_{\mathrm{T}}}

\DeclareDocumentCommand{\AgeIndex}{d[] d()}
{
\IfNoValueTF{#1}
	{
	\EM{\mathbb{I}(#2)}
	}
	{
	\EM{\mathbb{I}_{#1}(#2)}
	}
}
\DeclareDocumentCommand{\AgeStr}{d[] d()}
{
\IfNoValueTF{#1}
	{
	\EM{\cA_{#2}}
	}
	{
	\EM{\cA_{#1: #2}}
	}
}
\DeclareDocumentCommand{\AgeTuple}{d[] d()}
{
\IfNoValueTF{#1}
	{
	\EM{\aa_{#2}}
	}
	{
	\EM{\aa_{#1:#2}}
	}
}

\DeclareDocumentCommand{\EmbedInfo}{d[] d()}
{
\IfNoValueTF{#1}
	{
	\EM{\mathscr{EI}(#2)}
	}
	{
	\EM{\mathscr{EI}(#1, #2)}
	}
}

\DeclareDocumentCommand{\ClosureInfo}{d[] d()}
{
\IfNoValueTF{#1}
	{
	\EM{\mathscr{CI}(#2)}
	}
	{
	\EM{\mathscr{CI}(#1, #2)}
	}
}

\DeclareDocumentCommand{\Closure}{d<> d()}
{
\IfNoValueTF{#1}
{
    \IfNoValueTF{#2}
    {
        \EM{\mathrm{cl}}
    }
    {
        \EM{\mathrm{cl}(#2)}
    }
}
{
    \EM{\mathrm{cl}_{#1}(#2)}
}
}

\DeclareDocumentCommand{\bigClosure}{d<> d()}
{
\IfNoValueTF{#1}
{
    \EM{\mathrm{cl}\bigl(#2\bigr)}
}
{
    \EM{\mathrm{cl}_{#1}\bigl(#2\bigr)}
}
}

\DeclareDocumentCommand{\BigClosure}{d<> d()}
{
\IfNoValueTF{#1}
{
    \EM{\mathrm{cl}\Bigl(#2\Bigr)}
}
{
    \EM{\mathrm{cl}_{#1}\Bigl(#2\Bigr)}
}
}

\DeclareDocumentCommand{\ClosureMap}{d[] d()}
{
\EM{\mathfrak{t}_{\Closure}(#1, #2)}
}

\DeclareDocumentCommand{\clsim}{}
{
\EM{\simeq}
}

\DeclareDocumentCommand{\HP}{}%
{%
    \mathrm{(HP)}%
}
\DeclareDocumentCommand{\CHP}{d()}
{%
\IfNoValueTF{#1}%
{%
    \mathrm{(cHP)}%
}
{%
    \mathrm{(\EM{#1}\textrm{-}cHP)}%
}%
}

\DeclareDocumentCommand{\JEP}{}%
{%
    \mathrm{(JEP)}%
}%

\DeclareDocumentCommand{\CJEP}{d()}%
{%
\IfNoValueTF{#1}%
{%
    \mathrm{(cJEP)}%
}%
{%
    \mathrm{(\EM{#1}\textrm{-}cJEP)}%
}%
}%

\DeclareDocumentCommand{\AP}{}%
{%
    \mathrm{(AP)}%
}%
\DeclareDocumentCommand{\CAP}{d()}%
{%
\IfNoValueTF{#1}%
{%
    \mathrm{(cAP)}%
}%
{%
    \mathrm{(\EM{#1}\textrm{-}cAP)}%
}%
}%

\DeclareDocumentCommand{\coAP}{}%
{%
    \mathrm{(CAP)}%
}%
\DeclareDocumentCommand{\CcoAP}{d()}%
{%
\IfNoValueTF{#1}%
{%
    \mathrm{(cCAP)}%
}%
{%
    \mathrm{(\EM{#1}\textrm{-}cCAP)}%
}%
}%

\DeclareDocumentCommand{\WAP}{}%
{%
    \mathrm{(WAP)}%
}%
\DeclareDocumentCommand{\CWAP}{d()}%
{%
\IfNoValueTF{#1}%
{%
    \mathrm{(cWAP)}%
}%
{%
    \mathrm{(\EM{#1}\textrm{-}cWAP)}%
}%
}%

\DeclareDocumentCommand{\DisExt}{d()}
{%
\IfNoValueTF{#1}%
{%
    \mathrm{DE}
}%
{%
    \mathrm{DE_{#1}}
}%
}%

\DeclareDocumentCommand{\Embedding}{d|| d<> d[] d()}
{
\IfNoValueTF{#1}
	{
        \EM{E_{#2:#4:#3}} 
	}
	{
        \EM{E^{#1}_{#2:#4:#3}} 
 	}
}

\DeclareDocumentCommand{\CAT}{d()}
{
\IfNoValueTF{#1}
	{
        \EM{\mathfrak{Cat}} 
	}
	{
        \EM{\mathfrak{Cat}(#1)} 
 	}
}

\DeclareDocumentCommand{\WitcoAP}{d()}
{
\IfNoValueTF{#1}
{
    \mathrm{AB}
}
{
    \mathrm{AB_{#1}}
}
}

\DeclareDocumentCommand{\WitFuncoAP}{d()}
{
\IfNoValueTF{#1}
{
    \EM{f_{\WitcoAP}}
}
{
    \EM{f_{\WitcoAP(#1)}}
}
}

\DeclareDocumentCommand{\WitWAP}{d()}
{
\IfNoValueTF{#1}
{
    \EM{WAB}
}
{
    \EM{WAB(#1)}
}
}

\DeclareDocumentCommand{\WitFunWAP}{d()}
{
\IfNoValueTF{#1}
{
    \EM{f_{WAB}}
}
{
    \EM{f_{WAB(#1)}}
}
}

\makeatletter
\newcommand{\dotminus}{\mathbin{\text{\@dotminus}}}

\newcommand{\@dotminus}{%
  \ooalign{\hidewidth\raise1ex\hbox{.}\hidewidth\cr$\m@th-$\cr}%
}

\DeclareMathOperator{\Lang}{\mathcal{L}}
\DeclareMathOperator{\id}{id}

\def\aa{{\EM{\mathbf{a}}}}
\def\bb{{\EM{\mathbf{b}}}}
\def\cc{{\EM{\mathbf{c}}}}
\def\dd{{\EM{\mathbf{d}}}}
\def\ee{{\EM{\mathbf{e}}}}

\def\cA{{\EM{\mathcal{A}}}}
\def\cB{{\EM{\mathcal{B}}}}
\def\cC{{\EM{\mathcal{C}}}}
\def\cD{{\EM{\mathcal{D}}}}
\def\cE{{\EM{\mathcal{E}}}}

\def\cG{{\EM{\mathcal{G}}}}
\def\cK{{\EM{\mathcal{K}}}}
\def\cM{{\EM{\mathcal{M}}}}
\def\cN{{\EM{\mathcal{N}}}}

\def\w{\EM{\omega}}

\def\Integers{{\EM{{\mbb{Z}}}}}

\newcommand{\tconcat}{\ensuremath{{}^{\wedge}}}
\let\oldcaret\^
\DeclareRobustCommand{\^}{\ifmmode\tconcat\else\oldcaret\fi}

\def\Or{\EM{\vee}}
\def\And{\EM{\wedge}}

\def\<{\EM{\langle}}
\def\>{\EM{\rangle}}

\def\nl{\newline}

\def\EM#1{\ensuremath{#1}}
\def\mbb#1{\EM{\mathbb{#1}}}
\def\mbf#1{\EM{\mathop{\pmb{#1}}}}

\def\ol#1{\EM{\overline{#1}}}

\def\ul#1{\underline{#1}}

\def\st{\,:\,}
\def\:{\colon}

\providecommand{\dotdiv}{% Don't redefine it if available
  \mathbin{% We want a binary operation
    \vphantom{+}% The same height as a plus or minus
    \text{% Change size in sub/superscripts
      \mathsurround=0pt % To be on the safe side
      \ooalign{% Superimpose the two symbols
        \noalign{\kern-.35ex}% but the dot is raised a bit
        \hidewidth$\smash{\cdot}$\hidewidth\cr % Dot
        \noalign{\kern.35ex}% Backup for vertical alignment
        $-$\cr % Minus
      }%
    }%
  }%
}

\DeclareDocumentCommand{\RightJustify}{m}{\hspace*{\fill}\mbox{#1}\penalty-9999\relax}

%% file: theorems.tex
\DeclareDocumentCommand{\DeclareCounter}{m}%
{%
 	\ifcsname c@#1\endcsname%
	\else%
		\newcounter{#1}%
	\fi%
}

\DeclareDocumentCommand{\MyQED}{}{\qed}

\DeclareDocumentEnvironment{proof}{d() D[]{\MyQED} d<>}
{%BEGIN
\noindent\IfNoValueTF{#1}
{\emph{Proof.\!\!}}
{\emph{Proof\ #1.\ }}
}
{%END
\IfValueTF{#3}
	{%	
	\ExplSyntaxOff
	\RightJustify{\EM{#2_{#3}}}
	\vspace{1ex}
	}
	{
	\RightJustify{\EM{#2}}
	\vspace{1ex}
	}
}

\DeclareCounter{ProofLabelcOUntEr}
\DeclareDocumentCommand{\ProofLabel}{}{%
\addtocounter{ProofLabelcOUntEr}{1}
\label{cUrrEntProoflAbEl\arabic{ProofLabelcOUntEr}}
}

\DeclareDocumentCommand{\ProofRef}{D<>{1}}
{%
\ref{cUrrEntProoflAbEl\arabic{ProofcOUntEr#1}}
}
\DeclareDocumentCommand{\ProofCref}{D<>{1}}
{%
\cref{cUrrEntProoflAbEl\arabic{ProofcOUntEr#1}}
}

\DeclareDocumentEnvironment{Proof}{d() D[]{\MyQED} D<>{1}}
{%BEGIN
\DeclareCounter{ProofcOUntEr#3}%
\setcounter{ProofcOUntEr#3}{\value{ProofLabelcOUntEr}}%
\begin{proof}(#1)[#2]<\ProofRef<#3>>%
}
{%END
\end{proof}
}

\ExplSyntaxOn
\def\TheoremDepth{section}

\DeclareDocumentCommand{\DeclareTheorem}{m o m o}{%
\IfNoValueTF{#4}
	{%
	\IfNoValueTF{#2}
		{%
		\newtheorem{#1vArIAblE}{#3}
		}
		{%
		\newtheorem{#1vArIAblE}[#2vArIAblE]{#3}
		}
	}
	{%
	\newtheorem{#1vArIAblE}{#3}[#4]%
	}
\newtheorem*{#1vArIAblE*}{#3}

\DeclareDocumentEnvironment{#1}{o o}
	{%BEGIN
	\IfValueT{##2}%
		{
		\begin{spacing}{##2}
		}
	\IfValueTF{##1}
		{
		\begin{#1vArIAblE}[##1]
		}
		{
		\begin{#1vArIAblE}
		}
	\ProofLabel
	}
	{%END
	\IfValueT{##2}%
		{
		\end{spacing}{##2}
		}
	\end{#1vArIAblE}
	}

\DeclareDocumentEnvironment{#1*}{o o}
	{%BEGIN
	\IfValueT{##2}%
		{
		\begin{spacing}{##2}
		}
	\IfValueTF{##1}
		{
		\begin{#1vArIAblE*}[##1]
		}
		{
		\begin{#1vArIAblE*}
		}
	}
	{%END
	\IfValueT{##2}%
		{
		\end{spacing}{##2}
		}
	\end{#1vArIAblE*}
	}
}
\ExplSyntaxOff

%% file: ingredients/ComputableRepresentationsSection.tex
\section{Computable Representations}
\label{app:comp-model-theory}
%%%%  %%%%  %%%%  %%%%  %%%%  %%%%
Throughout this paper, we use standard notation from computable model theory for computable representations of structures.

We recall the definition of a computable representation of a language and a computable representation of a structure. 
\begin{definition}
Let $\Lang$ be a countable language. A \defn{computable representation} of $\Lang$ is a pair of maps $(\Rel_{\Lang}, \Funct_{\Lang})$ where 
\begin{itemize}
\item $\Rel_{\Lang}$ is a bijection from a computable subset $R$ of $\w$ to the set of relation symbols in $\Lang$, 

\item $\Funct_{\Lang}$ is a bijection from a computable subset $F$ of $\w$ to the set of function symbols in $\Lang$,

\item $R$ and $F$ are disjoint, and

\item the map $\arity_{\Lang}$ which takes an element of $R \cup F$ and returns the arity of the relation symbol or function symbol is computable. 
\end{itemize}
\end{definition}

\begin{definition}
Let $\Lang$ be a countable language with computable representation $(\Rel_{\Lang}, \Funct_{\Lang})$ and let $\cM$ be an $\Lang$-structure. A \defn{computable representation} of $\cM$ is a bijection $\iota_{\cA}$ from a c.e.\ subset $M$ of $\w$ to the underlying set of $\cM$ such that the following three sets are each c.e.:
%\begin{itemize}
%\item 
\[
\bigl\{(n, \aa) \st n \in \dom(\Rel_{\Lang}),\ \aa \in M, \text{ and }\cM \models \Rel_{\Lang}(n)(\iota_{\cM}(\aa))\bigr\}
\]
\[
\bigl\{(n, \aa) \st n \in \dom(\Rel_{\Lang}),\ \aa \in M, \text{ and }\cM \models \neg \Rel_{\Lang}(n)(\iota_{\cM}(\aa))\bigr\}
\]
%\item 
%\begin{eqnarray*}
\[
\bigl\{(f, \aa, b) \st f \in \dom(\Funct_{\Lang}),\ \aa \in M, \text{ and }
%\hspace*{12pt}\\
%\hspace*{12pt}
\cM \models \Funct_{\Lang}(f)(\iota_{\cM}(\aa)) = \iota_{\cM}(b)\bigr\}
\]
%\end{eqnarray*}
%\end{itemize}
%
We say that an $\Lang$-structure $\cM$ is \defn{computably enumerable} (c.e.) if the underlying set of $\cM$ is a c.e.\ subset of $\w$ and the identity map on this set is a computable representation of $\cM$. We say $\cM$ is \defn{computable} if the underlying subset of $\w$ is a computable subset. 
\end{definition}

The following is a standard result of computable model theory which shows that there is little difference between computable and c.e. $\Lang$-structures. 

\begin{lemma}
Uniformly in the index for a computable representation of a structure $\cM$ one can compute
%\begin{itemize}
%\item[(a)] 
a computable representation of $\cM$ whose underlying set is a coinfinite computable subset of $\w$, along with
%
%\item[(b)] 
an index for the underlying set of the new representation (as a computable subset of $\w$). 
%\end{itemize}
\end{lemma}

%% file: ingredients/CompLimitSequencesStructuresProof.tex
\begin{proof}
First for $i < j \in \w$ let $F_{i, j} = F_j \circ F_{j-1} \circ \dots \circ F_{i+1} \circ F_i$. We build a sequence $(\dd_i, \cD_i)_{i \in \w}$ of structures such that for all $i \in \w$, we have $\dd_i \in \cD_i \subseteq \cD_{i+1}$, and $\dd_i \clsim \AgeTuple[\compcK](k_{i})$, and also 
$\ClosureMap[\AgeTuple[\compcK](k_{i+1})](\dd_{i+1}) \circ \ClosureMap[\AgeTuple(k_i)](\cc_i) = \ClosureMap[\AgeTuple(k_i)](\dd_i)$. We do this by simultaneously enumerating all the elements $\AgeStr[\compcK](k_m)$ for $m \in \w$. When a new element $x$ appears in an $\AgeStr[\compcK](k_m)$ we first check to see if there is some $i<m$ such that $x$ is the image under $F_{i,m}$ of some element which already has been enumerated. If it is, then we do nothing. If it is not, then we check to see if there is some $j>m$ such that the image of $x$ under $F_{m, j}$ is some element $y$ which has already been enumerated. If so, then we add $y$ to $\cD_m$ and do nothing with $x$. If neither of these holds, then we add $x$ to $\cD_m$. 

Given a tuple $\xx$ in some $\cD_m$, there must be some $j$ such that each element of $\xx$ is in one of $\{\AgeStr[\compcK](k_0)\}_{s \in [j]}$. We can then use this fact along with the maps $F_{s, t}$ for $s, t \in [j]$ to determine which literals hold of $\xx$. 

It is then clear that $\dd_m \clsim \AgeStr[\compcK](k_m)$ for all $m\in\w$. Therefore we can find a computable age $\compcK^*$ which contains $\compcK$ along with $(\dd_i, \cD_i)_{i \in \w}$ and which further contains these collections in such a way that ensures that $\compcK$ and $\compcK^*$ are computably isomorphic. 
\end{proof}

%% file: ingredients/APimpliesJumpProof.tex
\begin{proof}
We describe an algorithm which uses
$\TuringDegree[\EmbedInfo(\compcK)]$ 
to map a potential span $(F_0, F_1)$ to an amalgamation diagram over it. \nl\nl
\ul{Step 1:}\nl
Use $\TuringDegree[\EmbedInfo(\compcK)]$ 
to check whether each of $F_0$ and $F_1$ is an embedding. \nl\nl
\ul{Step 2, Case a: at least one of $F_0$ or $F_1$ is not an embedding} \nl 
Let $i$ be such that $\AgeIndex(i) = \codom(F_0)$. Let $I = (\AgeIndex(i), \AgeIndex(i), \AgeTuple(i))$ be the identity map on $\AgeIndex(i)$. Let $G = (\codom(F_1), \AgeIndex(i), \AgeTuple(i))$. 
Output the pair $(G, I)$, which is an amalgamation diagram over $(F_0, F_1)$. 
\nl\nl%     
\ul{Step 2, Case b: both $F_0$ and $F_1$ are embeddings} \nl 
Use $\TuringDegree[\EmbedInfo(\compcK)]$ 
to search for a pair $(G_0, G_1)$ which is an amalgamation diagram over $(F_0, F_1)$. Because $\compcK$ satisfies $\AP$ we will always find such an amalgamation diagram. Output the pair $(G_0, G_1)$.
%\nl\nl 
%
\end{proof}

%% file: ingredients/UpperBoundCompCAPProof.tex
\begin{proof}
We will first define 
$\TuringDoubleJump[\EmbedInfo(\compcK)]$-computable 
functions $h_1$, $h_2$ and $h_3$. 
Let $h_1$ be the function which takes as input a tuple $(F, G_0, G_1, H_0, H_1)$ where 
\begin{itemize}
\item $F, G_0, G_1, H_0, H_1$ are potential embeddings, 

\item $\codom(F) = \dom(G_0) = \dom(G_1)$, 

\item $\codom(H_0) = \codom(H_1)$,  and

\item  $\dom(H_i) = \codom(G_i)$  for $i \in \{0, 1\}$, 
\end{itemize}
and which returns $1$ as output if both
\begin{itemize}
\item $F$ is an embedding and

\item $(H_0, H_1)$ is an amalgamation diagram over $(G_0, G_1)$,
\end{itemize}
and returns $0$ as output otherwise. 

Note that $h_1$ is $\EmbedInfo(\compcK)$-computable as we can determine from $\EmbedInfo(\compcK)$ whether or not a potential embedding is an embedding and hence we can also determine from $\EmbedInfo(\compcK)$ whether or not $(H_0, H_1)$ is an amalgamation diagram over $(G_0, G_1)$. 

Let $h_2$ be the function which takes as input a triple $(F, G_0, G_1)$ of potential embeddings where $\codom(F) = \dom(G_0) = \dom(G_1)$, and which returns $1$ if there exists $(H_0, H_1)$ such that $h_1(F, G_0, G_1, H_0,H_1)= 1$, and returns $0$ otherwise.

Observe that $h_2$ holds of a specific potential span $(G_0, G_1)$ along with a potential embedding $F$ if and only if there is an amalgamation diagram over $(G_0, G_1)$ for which $F$ is an embedding into $\dom(G_0)$. Note that $h_2$ is 
$h_1'$-computable and hence is $\TuringJump[\EmbedInfo(\compcK)]$-computable.

Finally, let $h_3$ be the function which takes as input a potential embedding $F$, and which returns $1$ if both
\begin{itemize}
\item $F$ is an embedding and 
\item $h_2(F,G_0, G_1) = 1$
whenever $(G_0, G_1)$ is a potential span with $\dom(G_0) = \dom(G_1) = \codom(F)$, 
\end{itemize}
and returns $0$ otherwise. 

Observe that $h_3$ holds of $F$ precisely when the codomain of $F$ is an amalgamation base. Note that $h_3$ is 
$h_2'$-computable 
and hence is $\TuringDoubleJump[\EmbedInfo(\compcK)]$-computable. 

We first prove (b). Note that for any $\cM \in \compcK$, we have $\id_{\cM} \in \WitcoAP$ (i.e., $\cM$ is an amalgamation base) precisely when $h_3(\id_{\cM}) = 1$ holds.  Hence $\cM$ is computable from $h_3$, and so
$\TuringDegree[\WitcoAP] \Turingleq \TuringDoubleJump[\EmbedInfo(\compcK)]$.

To prove (a), let $\WitFuncoAP$ be the function which takes as input a tuple $(F, G_0, G_1)$ where $F \in \WitcoAP$ and where $(G_0, G_1)$ is a potential span with $\dom(G_0) = \dom(G_1) = \codom(F)$, and which uses $h_1$ to search for an amalgamation diagram over $(G_0, G_1)$ and output the first one it finds. Because $\WitcoAP$ consists of maps whose codomain is an amalgamation base, such an amalgamation diagram always exists, and this function will always converge. Now suppose $\compcK$ has $\coAP$. Then $\WitcoAP$ is a collection of distinguished extensions and hence $(\WitcoAP, \WitFuncoAP)$ is a witness to $\coAP$. However $(\WitcoAP, \WitFuncoAP)$ is $h_3$-computable and hence $\TuringDoubleJump[\EmbedInfo(\compcK)]$-computable. Therefore $\compcK$ has  $\CcoAP(\TuringDoubleJump[\EmbedInfo(\compcK)])$.
\end{proof}

%% file: ingredients/cCompUltracCapProof.tex
\begin{proof}
Let $g$ be a witness to the $\StandardTuringDegree$-computable ultrahomogeneity of $\cM$. 
Suppose $(F_0, F_1)$ is a potential span over $\AgeIndex(i)$, and suppose $h_0 = g(\AgeTuple(i), \range(F_0))$ and $h_1 = g(\AgeTuple(i), \range(F_1))$.
Let $G_0$ be the inclusion map from $\codom(F_0)$ to $\codom(F_0) \cup h_0h_1^{-1}(\codom(F_1))$ and let $G_1$ be the map $h_0h_1^{-1}$ composed with the inclusion into $\codom(G_0)$. 

It is immediate that $G_0$ is an embedding. Further, $G_1$ is an embedding if $F_1$ is.
Therefore $(G_0, G_1)$ is an amalgamation diagram over $(F_0, F_1)$. The maps $G_0$ and $G_1$ are each computable from the two maps $h_0$ and $h_1$, which are themselves each computable from $(F_0, F_1)$ and $g$. Therefore $\compcK[\cM]$ has $\CAP(\StandardTuringDegree)$. 
\end{proof}

%% file: ingredients/UltraHomogeneousBoundProof.tex
\begin{proof}
Let $(x_i)_{i \in \w}$ be an enumeration (possibly with repetitions) of $\cM$. Let $\aa, \bb \in \cM$ be tuples with $\aa \tuplesim \bb$.

First $\EmbedInfo(\cM)$-computably check whether or not $\aa \clsim \bb$. \nl\nl
\ul{\textbf{Trivial Case:}}\nl
If $\aa \not \clsim \bb$, then let $g_{\aa, \bb}$ be any bijection extending the map which takes the underlying set of $\aa$ to the underlying set of $\bb$ in a way that $g_{\aa, \bb}(\aa) = \bb$. Note that we can always find such a map as $\aa \tuplesim \bb$.  \nl\nl
\ul{\textbf{Non-Trivial Case:}}\nl
If $\aa \clsim \bb$ we define the isomorphism $g_{\aa, \bb}\:\cM \to \cM$ in stages. \nl\nl
\ul{Stage $0$:} \nl
Let $g^0_{\aa, \bb} = \ClosureMap[\aa](\bb)\rest[\aa]$. \nl\nl
\ul{Stage $2k+1$}. \nl
If $x_k \in \range(g_{\aa, \bb}^{2k})$, then let $g_{\aa, \bb}^{2k+1} = g_{\aa, \bb}^{2k}$.
Otherwise, let $\cc_{2k}$ be an enumeration of $\dom(g_{\aa, \bb}^{2k})$ and let $\dd_{2k}$ be an enumeration of $\range(g_{\aa, \bb}^{2k})$.
Search for a $\beta_{2k}\in \w$ such that $\cc_{2k} \tconcat x_{\beta_{2k}}  \clsim \dd_{2k} \tconcat x_{2k}$. We know such a $\beta_{2k}$ exists and hence we can 
$\TuringDegree[\EmbedInfo(\cM)]$-computably find one such value. 
Let $g_{\aa, \bb}^{2k+1} = \ClosureMap[\cc_n\tconcat x_{\beta_{2k}}](\dd_{2k}\tconcat x_{2k})$. \nl\nl
\ul{Stage $2k+2$}. \nl 
If $x_k \in \dom(g^{2k+1}_{\aa, \bb})$, then let $g_{\aa, \bb}^{2k+2} = g_{\aa, \bb}^{2k+1}$.
Otherwise, let $\cc_{2k+1}$ be an enumeration of $\dom(g_{\aa, \bb}^{2k+1})$ and let $\dd_{2k+1}$ be an enumeration of $\range(g_{\aa, \bb}^{2k+1})$. 
Search for an $\alpha_{2k+1} \in \w$ such that $\cc_{2k+1} \tconcat x_k  \clsim \dd_{2k+1} \tconcat x_{\alpha_{2k+1}}$. We know such a $\alpha_{2k+1}$ exists as $\cM$ is ultrahomogeneous. We can therefore 
$\TuringDegree[\EmbedInfo(\cM)]$-computably find one such value.\\
Let $g_{\aa, \bb}^{2k+2} = \ClosureMap[\cc_{2k+1}\tconcat x_{2k+1}](\dd_{2k+1}\tconcat x_{\alpha_{2k+1}})$. \nl\nl
Let $g_{\aa, \bb} = \bigcup_{n \in \w} g_{\aa, \bb}^n$. Clearly $g_{\aa, \bb}$ is uniformly 
$\TuringDegree[\EmbedInfo(\cM)]$-computable in $\aa, \bb$ 
and witnesses the $\TuringDegree[\EmbedInfo(\cM)]$-ultrahomogeneity of $\cM$.  
\end{proof}

%% file: ingredients/cofIffScompForSomeSProof.tex
\begin{proof}
The implication from (b) to (a) is immediate.  

Suppose (a) holds. Let $E$ be the collection of tuples $\aa$ in $\cM$ such that whenever $f\:\Closure<\cM>(\aa) \to \cA^*$ is an isomorphism then there is an automorphism of $\cM$ extending $f$. 

We now define $g_{\aa,\bb}$ 
for all $\aa$, $\bb$ in $\cM$ with $\aa \in E$ and $\aa \tuplesim \bb$. Let $t_{\aa, \bb}$ be the unique bijection from the underlying set of $\aa$ to the underlying set of $\bb$ such that $t_{\aa, \bb}(\aa) = \bb$; such a $t_{\aa, \bb}$ must exist as $\aa \tuplesim \bb$. 
\begin{itemize}
\item If $\aa \clsim \bb$ then let $g_{\aa, \bb}$ be any automorphism of $\cM$ extending $t_{\aa, \bb}$.

\item If $\aa \not \clsim \bb$ then let $g_{\aa, \bb}$ be any bijection from the underlying set of $\cM$ to itself extending $t_{\aa, \bb}$.
\end{itemize}

Let $\StandardTuringDegree$ be the join of the Turing degree of $E$ with the Turing degree of  $\{g_{\aa, \bb} \st \aa \in E\text{ and }\bb \tuplesim \aa\}$. Then $\cM$ is $\StandardTuringDegree$-computably cofinally ultrahomogeneous and (b) holds. 
\end{proof}

%% file: ingredients/cofUltraImpliescofUltraAmalgamationBasesProof.tex
\begin{proof}
Let $\aa \in E$ and suppose $(G_0, G_1)$ is a span for which $\dom(G_0) = \Closure<\cM>(\aa)$.
We must show that there is some amalgamation diagram over $(G_0, G_1)$.
Note that there are embeddings $J_0\: \codom(G_0) \to \cM$ and $J_1\: \codom(G_1) \to \cM$. Let $K_i = J_i \circ G_i$ for $i \in \{0, 1\}$. Then $K_0$ and $K_1$ are each embeddings of 
$\Closure<\cM>(\aa)$ into $\cM$.
Therefore $K_1 \circ K_0^{-1}$ is an isomorphism from 
$\range(K_0)$ to $\range(K_1)$.

Because 
$\aa \in E$,
there is an automorphism $\alpha$ of $\cM$ such that
$\alpha \rest[\range(K_0)] = K_1 \circ K_0^{-1}$.
Let $\cC = \bigClosure<\cM>({\range(J_1)\cup \alpha``(\range(J_0))})$. Let $I\:\range(J_1) \to \cC$ be the inclusion map. Then $(\alpha \circ J_0, I \circ J_1)$ is an amalgamation diagram over $(G_0, G_1)$, and so
$\Closure<\cM>(\aa)$ is an amalgamation base. 
\end{proof}

%% file: ingredients/cofUltraIsoExtendsAutomorphismProof.tex
\begin{proof} 
For each $j \in \{0, 1\}$ let $\cB_j \supseteq \cA_j$ be such that any embedding from $\cB_j$ to $\cM$ can be extended to some automorphism of $\cM$. Because $\cA_0$ is an amalgamation base there must be a finitely generated $\cC \subseteq \cM$ along with embeddings $\alpha_0\:\cB_0 \to \cC$ and $\alpha_1\:\cB_1 \to \cC$ such that $\alpha_0\rest[\cA_0] = \alpha_1 \circ k$. 

Then for each $j \in \{0, 1\}$, the definition of $\cB_j$ tells us that there is an automorphism $\beta_j$ of $\cM$ that  extends $\alpha_j$.
Hence $\beta_1^{-1}\circ \beta_0$ is an automorphism of $\cM$  that extends $k$. 
\end{proof}

%% file: ingredients/cCompCofImpcCcoAPProof.tex
\begin{proof}
Suppose $\TuringDegree[\EmbedInfo(\cM)] \Turingleq \StandardTuringDegree$.
We need to construct an $\StandardTuringDegree$-computable witness to $\coAP$.  Let $E$ be an $\StandardTuringDegree$-computable cofinal collection that witnesses the $\StandardTuringDegree$-computably cofinal ultrahomogeneity of $\cM$.  Let  $\WitcoAP(\compcK[\cM])$ be the collection of embeddings between elements of $\compcK[\cM]$ whose codomain is isomorphic to an element of $E$. 
Note that  $\WitcoAP(\compcK[\cM])$ is computable from $\TuringDegree[\EmbedInfo(\cM)]$ and $E$, each of which is
$\StandardTuringDegree$-computable. 

Now let $(G_0, G_1)$ be a potential span in $\compcK[\cM]$ and suppose $F\in \WitcoAP(\compcK[\cM])$ with $\codom(F) = \dom(G_0)$.
We will $\StandardTuringDegree$-computably find an amalgamation diagram $(H_0, H_1)$. There are two cases.

If $(G_0, G_1)$ is not a span, then let $(H_0, H_1)$ be a pair of maps such that $\dom(H_0) = \codom(G_0)$ and $\dom(H_1) = \codom(G_1)$, and also such that $\codom(H_0) = \codom(H_1)$. Note that such a pair of maps always exists as $\compcK[\cM]$ has $\JEP$. We can therefore $\StandardTuringDegree$-computably find such a pair by searching for it using $\TuringDegree[\EmbedInfo(\cM)]$. 

If $(G_0, G_1)$ is a span, then let $I_0\: \codom(G_0) \to \cM$ and $I_1\:\codom(G_1) \to \cM$ be embeddings. Note that such embeddings always exist as $G_0$ and $G_1$ are embeddings in $\compcK(\cM)$. For $i \in \{0, 1\}$ let $J_i = I_i \circ G_i$.  As $(G_0, G_1)$ is a span and $\dom(G_0) = \dom(G_1) \in E$ there must be an automorphism $\alpha$ of $\cM$ which extends $J_1^{-1} \circ J_0$. Further, as $\cM$ is $\StandardTuringDegree$-computably cofinally ultrahomogeneous, we can $\StandardTuringDegree$-computably find one such automorphism $\alpha$, uniformly in $(F, G_0, G_1)$. But then $(\alpha \circ I_0, I_1)$ is an amalgamation diagram over $(G_0, G_1)$. Let $(H_0, H_1) = (\alpha \circ I_0, I_1)$. 

Let $\WitFuncoAP(\compcK[\cM])$ be the map which takes $(F, G_0, G_1)$ to $(H_0, H_1)$. Note that $\WitFuncoAP(\compcK[\cM])$ is $\StandardTuringDegree$-computable because $\TuringDegree[\EmbedInfo(\cM)]$ can be used to determine whether or not $(G_0, G_1)$ is a span, and each of the two cases $\StandardTuringDegree$-computably finds $(H_0, H_1)$.
Therefore $(\WitcoAP(\compcK[\cM]), \WitFuncoAP(\compcK[\cM]))$ is a witness to $\compcK[\cM]$ having $\CcoAP(\StandardTuringDegree)$, completing the proof.
%\qed\end{proofof}
\end{proof}

%% file: ingredients/EquivBetweenCofinalProof.tex
\begin{proof}
Suppose that $\cM$ is 
$\StandardTuringDegree$-computably cofinally ultrahomogeneous with respect to $E$, and let
$h$ be a function witnessing this.
Write $h_{\aa, \bb}$ for function encoded by $h(\aa, \bb)$. Suppose $\aa, \bb, \cc \in \cM$ with $\aa, \cc \in E$. If $\aa \clsim \bb$ then $h_{\aa, \bb}$ is  an automorphism of $\cM$ that takes $\aa$ to $\bb$.
Let $p(\aa, \bb, \cc) = h_{\aa, \bb}(\cc)$. Then $E$ and $p$ witness 
the $\StandardTuringDegree$-computable cofinal extension property for $\cM$ with respect to $E$.
Further, as $E$ and $h$ are $\StandardTuringDegree$-computable, 
$p$ is also $\StandardTuringDegree$-computable.  

For the other direction, now suppose that $\cM$ has the 
$\StandardTuringDegree$-computable cofinal extension property for $\cM$ with respect to $E$, and let
$p$ be a function witnessing this.
Let $(x_i)_{i \in \w}$ be an $\StandardTuringDegree$-computable enumeration of the domain of $\cM$. Suppose $\aa, \bb \in \cM$ with $\aa \in E$. 
We will define by induction a sequence of functions $(h^n_{\aa, \bb})_{n\in \w}$ and sequences of tuples $(\aa_n)_{n\in\w}$ and $(\bb_n)_{n\in\w}$. These will be defined such that 
for each $n\in\w$, the following inductive hypotheses hold.
\begin{itemize}
\item $h_{\aa, \bb}^n \subseteq h_{\aa, \bb}^{n+1}$.
\item $\aa_n$ is an enumeration of the domain of $h^n_{\aa, \bb}$.
\item $\bb_n = h^n_{\aa, \bb}(\aa_n)$. 
\item $\aa_n, \bb_n \in E$ when $\aa \clsim \bb$.
\end{itemize}
%\nl\nl
\ul{Base Case: $n = 0$}\nl  
We let $h_{\aa, \bb}^0$ be the unique bijection from $\aa$ to $\bb$, which must exist because $\aa \tuplesim \bb$. Let $\aa_0 =\aa$ and $\bb_0 = \bb$. Note that $\aa_0 \in E$ by hypothesis, and that $\bb_0\in E$ when $\aa \clsim \bb$.\nl\nl
\ul{Trivial Case: $n > 0$ and $\aa \not \clsim \bb$}\nl
Let $\aa_n = \aa \tconcat \<x_i\>_{i < n}$.
Let $h^n_{\aa, \bb}$ be any bijection with domain $\aa_n$ and codomain contained in $\cM$ which extends $h^{n-1}_{\aa, \bb}$.  
Let $\bb_n = h^n_{\aa, \bb}(\aa_n)$.
\nl\nl
\ul{Inductive Case: $n = 2k+1$ for some $k\ge 0$ (and not Trivial Case):}\nl
If $x_k \in \dom(h^n_{\aa, \bb})$ then let $h^{n+1}_{\aa, \bb} = h^n_{\aa, \bb}$,  and let $\aa_{n+1} = \aa_n$ and $\bb_{n+1} = \bb_n$. Now suppose $x_k \not \in \dom(h^n_{\aa, \bb})$. 
Let $\aa_{n+1}^* \in E$ be the first (relative to the $\StandardTuringDegree$-computable enumeration of $E$) tuple such that $\aa_n\tconcat x_k \in \Closure(\aa_{n+1}^*)$. Let $\aa_{n+1} = \aa_n \tconcat x_k\tconcat\aa_{n+1}^*$, and note that $\aa_{n+1} \in E$ as well. Let $\bb_{n+1} = p(\aa_n, \bb_n, \aa_{n+1})$. Note that by the inductive hypothesis, $\aa_n \clsim \bb_n$, which implies that $\aa_{n+1} \clsim \bb_{n+1}$ and hence $\bb_{n+1} \in E$. Let $h^{n+1}_{\aa, \bb}$ be the bijection which takes $\aa_{n+1}$ to $\bb_{n+1}$. 
Note that $h^n_{\aa, \bb} \subseteq h^{n+1}_{\aa, \bb}$  because $\aa_n$ is a subtuple of $\aa_{n+1}$.   \nl\nl
\ul{Inductive Case: $n = 2k+2$ for some $k \ge 0$ (and not Trivial Case):}\nl 
If $x_k \in \range(h^n_{\aa, \bb})$ then let $h^{n+1}_{\aa, \bb} = h^n_{\aa, \bb}$ and let $\aa_{n+1} = \aa_n$ and $\bb_{n+1} = \bb_n$. Now suppose $x_k \not \in \range(h^n_{\aa, \bb})$. Let $\bb_{n+1}^* \in E$ be the first (relative to the $\StandardTuringDegree$-computable enumeration of $E$) tuple such that $\bb_n\tconcat x_k \in \Closure(\bb_{n+1}^*)$. Let $\bb_{n+1} = \bb_n \tconcat x_k\tconcat\bb_{n+1}^*$, and note that $\bb_{n+1} \in E$ as well. Let $\aa_{n+1} = p(\bb_n, \aa_n, \bb_{n+1})$. Note that by the inductive hypothesis, $\aa_n \clsim \bb_n$, which implies that  $\aa_{n+1} \clsim \bb_{n+1}$ and hence $\aa_{n+1} \in E$. Let $h^{n+1}_{\aa, \bb}$ be the bijection which takes $\aa_{n+1}$ to $\bb_{n+1}$. 
Note that $h^n_{\aa, \bb} \subseteq h^{n+1}_{\aa, \bb}$  because $\bb_n$ is a subtuple of $\bb_{n+1}$.   \nl\nl
Let $h_{\aa, \bb} = \bigcup_{n \in \w} h^n_{\aa, \bb}$. Note that whenever $\aa \clsim \bb$, then by induction we have $\aa_n \clsim \bb_n$ for all $n \in \w$, and hence $h_{\aa, \bb}$ is an automorphism of $\cM$. Further, $h_{\aa, \bb}$ is $\StandardTuringDegree$-computable uniformly from $(\aa, \bb)$ because $p$ is $\StandardTuringDegree$-computable and we can compute from $\EmbedInfo(\cM)$ whether or not each $n$ is in the Trivial Case. 
\end{proof}

%% file: ingredients/compCofFrLimitProof.tex
\begin{proof}
\emph{(Under a mild technical assumption, which we now describe.)}
\cref{Computable limit of sequences of structure} 
and \cref{cor:union-of-comp-chain} transform
a sequence of embeddings into an increasing sequence of structures with respective elements isomorphic, such that the union
of the sequence produces an infinite computable structure.
In this proof we use this technique to build a cofinal \Fraisse\ limit. We do so while requiring that for each an element $\cB_n$ in the sequence, there is an appropriate substructure $\cA \subseteq \cB_n$ and an embedding from $\cA$ into $\cC$ that we can amalgamate the inclusion map with the embedding. In order to do this we will need to enumerate all finite subsets of each $\cB_n$ in our sequence. Because in general our amalgamations only give us embeddings and not inclusion maps, there is a non-trivial amount of bookkeeping needed. This bookkeeping, though annoying, is not difficult. As such we will assume in this proof that when we have an amalgamation diagram one element of the diagram is an inclusion and not just an embedding. 

As $\compcK$ satisfies $\CcoAP(\StandardTuringDegree)$, there is an $\StandardTuringDegree$-computable witness to $\coAP$. Let $(\WitcoAP, \WitFuncoAP)$ be one such.

We will now construct an increasing sequence 
$(\bb_n, \cB_n, \ell_n)_{n \in \w}$ 
of elements of $\compcK$,  such that
$\cB_n \subseteq \cB_{n+1}$
for all $n \in \w$. We also let $\cB_{n, k}$ be the result of applying at most $k$ functions to elements in $\bb_n$. Note that $\cB_n = \bigcup_{k \in \w} \cB_{n, k}$. 

Without loss of generality we can assume that each $\cB_n$ has as its domain a subset of $\w$. We also let $(\dd_n, F_n, G_n)_{n \in \w}$ be an enumeration, with infinite repetitions, of triples where 
\begin{itemize}
\item $\dd_n$ is a finite subset of $\w$, 

\item  $F_n\in \WitcoAP$,  and

\item $G_n$ is an embedding with $\dom(G_n) = \codom(F_n)$. Let $\dom(G_n) = \AgeIndex(j_n)$ where $j_n \in \w$. 
\end{itemize}

Note that we can find a sequence which is $\StandardTuringDegree$-computable. We will now build our sequence in stages. 
\nl\nl
\ul{Stage $0$}:\nl
Because $\compcK$ satisfies $\JEP$ there is a unique element (up to isomorphism) generated by the empty tuple. Let $\cc$ be the empty tuple, and let $(\cc, \cC, i)\in \compcK$ be an element that $\cc$ generates. Let $F_0$ be such that $F_0 \in \WitcoAP(\compcK)$ and $\dom(F_0) = (\cc, \cC, i)$. We know one such exists because $\WitcoAP(\compcK)$ is a collection of distinguished extensions. Let $\cB_0$ be the structure in the codomain of $F_0$. \nl\nl
\ul{Stage $n+1$:}\nl
\ul{Case $a$:} Either 
\begin{itemize}
\item $\dd_n \not \subseteq \bigcup_{i \in [n+1]} \cB_{i, n}$, or 

\item $\dd_n \subseteq \bigcup_{i \in [n+1]} \cB_{i, n}$ but $\dd_n \not \clsim \AgeIndex(j_n)$.
\end{itemize} 
We let $(\bb_{n+1}, \cB_{n+1}, \ell_{n+1}) = (\bb_n, \cB_{n}, \ell_n)$. \nl\nl
\ul{Case $b$:} Otherwise. \nl
Let $\cD_n = \Closure(\dd_n)$ and $\ClosureMap[{\dd_n}](\AgeStr(j_n))\: \cD_n \to \AgeStr(j_n)$ is an isomorphism. Let $K_n = G_n \circ \ClosureMap[\dd_n](\AgeStr(j_n))$.  Let $I_n$ be the inclusion map from $\cD_n$ into $\cB_n$. We then have $(I_n, K_n)$ is a span. Let $(H_0, H_1) = \WitFuncoAP(I_n, K_n)$. Then $(H^n_0, H^n_1)$ is an amalgamation diagram over $(I_n, K_n)$. We then let $(\bb_{n+1}, \cB_{n+1}, \ell_{n+1}) = \codom(H^n_0)$. 

Note that we are assuming (as discussed at the beginning of the proof) that $H^n_0$ is an inclusion map and hence $\cB_n \subseteq \cB_{n+1}$. \nl\nl
We now let $\cB_\w = \bigcup_{n \in \w}\cB_n$. Note that because we can determine from $\StandardTuringDegree$ which case we are in and compute $\WitFuncoAP$, $\cB_\w$ is $\StandardTuringDegree$-computable. It therefore suffices to prove the following two claims.

\begin{claim}
$\compcK[\cB_\w]$ is $\StandardTuringDegree$-computably isomorphic to $\compcK$. 
\end{claim}
\begin{claimproof}
As $\compcK$ satisfies $\CHP(\StandardTuringDegree)$ we $\compcK[\cB_\w]$ is an $\StandardTuringDegree$-computable subset of $\compcK$. 

Now suppose $\AgeIndex[\compcK](i) \in \compcK$. As $\compcK$ has the $\CJEP(\StandardTuringDegree)$ we can find, computably in $\StandardTuringDegree$, embedding $M_0\: (\bb_{0}, \cB_{0}, \ell_{0}) \to \AgeIndex[\compcK](k)$ and $M_1\: \AgeIndex[\compcK](i) \to \AgeIndex[\compcK](k)$ for some $k \in \w$. For each $\ell \in \w$, let $J_\ell$ be the inclusion map from $(\bb_{0}, \cB_{0}, \ell_{0})$ to $(\bb_{\ell}, \cB_{\ell}, \ell_{\ell})$. Then there must be some $n$ such that $(I_n, K_n) = (J_n, M_0)$. Therefore, as $\cB_{n+1}$ is an amalgamation of $(J_n, M_0)$ there is an embedding of $\AgeIndex[\compcK](i)$ into $\cB_{n+1}$. Because $\compcK[\cB_\w]$ has $\CHP(\StandardTuringDegree)$ this gives us an $\StandardTuringDegree$-computable map from $\compcK$ to $\compcK[\cB_\w]$. 

It is straightforward to check that this map, along with the inclusion from $\compcK[\cB_\w]$ into $\compcK$, give us the desired $\StandardTuringDegree$-computable isomorphism. 
\end{claimproof}

\begin{claim}
$\cB_\w$ is $\StandardTuringDegree$-computably cofinally ultrahomogeneous. 
\end{claim}
\begin{claimproof}
Let $E$ be the collection of tuples $\aa\in \cB_\w$ such that $\aa = \aa_0 \aa_1$ where $\aa_0 \in \Closure(\aa_1)$ for some $n \in \w$ we have $\aa_1 \clsim \bb_n$. Note that $E$ is $\StandardTuringDegree$-computable as it is computable from $(\cB_n)_{n \in \w}$ and $\EmbedInfo(\compcK)$. Further it is immediate that $E$ is a cofinal collection in $\cB_\w$. 

By \cref{Equivalence between cofinal ultrahomogeneity and cofinal extension property} it suffices to show that $\cB_\w$ has the cofinal extension property with respect to $E$. We now define an $\StandardTuringDegree$-computable function $p$ witnessing this fact. Suppose $(\aa, \bb, \cc)$ with $\aa, \bb, \cc \in \cB_\w$, $\aa, \cc \in E$, $\aa \subseteq \Closure[\cB_\w](\cc)$ and $\aa \tuplesim \bb$. We now break into two cases. \nl\nl
\ul{Case 1:} $\aa \not \clsim \cc$\nl
We let $p(\aa, \cc, \dd) = \ee$ be any tuple such that $\aa \cc \tuplesim \dd \ee$. \nl\nl
\ul{Case 2:} $\aa \clsim \cc$. \nl
Let $n_0$ be such that $\cc \clsim \bb_{n_0}$ and let $n_1$ be such that $\cc \subseteq \bb_{n_1}$. Let $I_{\aa, \dd}$ be the inclusion map from $\Closure(\aa)$ into $\Closure(\dd)$ and let $J = I_{\aa, \dd} \circ \ClosureMap[\cc](\aa)$. Because of how $(\cB_n)_{n \in \w}$ was defined there must some $n^* \geq n_1$ such that if $I_{n_1, n^*}$ is the inclusion map from $\cB_{n_1}$ into $\cB_{n^*}$ then there is an amalgamation diagram $(H_0, H_1)$ over $(I_{n_1, n^*}, J)$ where $\codom(H_0) = (\bb_{n^*+1}, \cB_{n^*+1}, \ell_{n^*+1})$. But then $H_1$ is an isomorphism with domain $\Closure(\dd)$ where $H_1 \circ J(\cc) = \cc$. Therefore if $\ee = H_1(\dd)$ with we have $\ee \clsim \dd$ and $\ClosureMap[\dd](\ee) (\aa) = \cc$. Therefore we can let $p(\aa, \cc, \dd) = \ee$. Note $p$ is therefore $\StandardTuringDegree$-computable as the construction of $(\cB_n)_{n \in \w}$ is and $\EmbedInfo(\compcK) \Turingleq \StandardTuringDegree$. 
\end{claimproof}

This completes the proof of \cref{prop:computable-cofinal-Fr-limit}.
\end{proof}

%% file: ingredients/cofUltrDoubleJumpProof.tex
\begin{proof}    
Let $E$ be the collection of tuples $\aa$ such that $\Closure<\cM>(\aa)$ is an amalgamation base. Note that $E$ is $\TuringDoubleJump[\EmbedInfo(\cM)]$-computable by \cref{Upper bounds on the computability of (CAP)}~(b).  

Because $\cM$ is cofinally ultrahomogeneous, there is some
cofinal collection $E_0$ in $\cM$ such that $\cM$ is cofinally ultrahomogeneous for $E_0$.
By \cref{Cofinal ultrahomogeneous implies cofinal ultrahomogeneous with respect to amalgamation bases}, for each $\aa \in E_0$ the substructure $\Closure<\cM>(\aa)$ is an amalgamation base. Therefore $E_0 \subseteq E$ and hence $E$ is a cofinal collection.

By \cref{Cofinal ultrahomogeneous implies isomorphism of amalgamation bases extend to automorphisms}, any isomorphism of amalgamation bases can be extended to some automorphism of $\cM$. Suppose $\cA_0$ and $\cA_1$ are amalgamation bases that are isomorphic via a map $j$ and suppose $\cA_0 \subseteq \cB_0$ for some amalgamation base $\cB_0$. Then there is some  $\cB_1$ with $\cA_1 \subseteq \cB_1$ for which there is an isomorphism from $\cB_0$ to $\cB_1$ that extends the isomorphism $j$. Therefore there is some function $p$ computable from $E$ and $\EmbedInfo(\cM)$ which witnesses that $\cM$ has the $\TuringDoubleJump[\EmbedInfo(\cM)]$-computable cofinal extension property. 

We also know that $\cM$ is computable from $\EmbedInfo(\cM)$ (and hence in particular is \linebreak $\TuringDoubleJump[\EmbedInfo(\cM)]$-computable). Therefore by \cref{Equivalence between cofinal ultrahomogeneity and cofinal extension property}, $\cM$ is $\TuringDoubleJump[\EmbedInfo(\cM)]$-computably cofinally ultrahomogeneous. 
\end{proof}

%% file: ingredients/LowerBoundCoAPFiniteProof.tex
\begin{proof}
Let $\Lang_0 = \{C, E\}$ where $C, E$ are binary. Let $\cM_0$ be the structure where 
\begin{itemize}
\item the underlying set is $M_0 \subseteq \w$, 

\item $(M, C^{\cM_0})$ is a directed graph, 

\item $\cM_0 \models (\forall x)(\exists^{=1} y)\, C(x, y) \And (\exists^{=1} z)\, C(z, x)$, 

\item $\cM_0$ has infinitely many $C$-cycles of every finite length,

\item $E^{\cM_0}$ is an equivalence relation on $M_0$, 

\item if $\cM \models E(a,b) \And C(a, a^*) \And C(b, b^*)$ then $\cM \models E(a^*, b^*)$, 

\item if $A_0, A_1$ are both $C$-cycles of length $n$ then for every $a_0 \in A_0$ there is a unique $a_1 \in A_1$ such that $E(a_0, a_1)$ holds, and 

\item if $A$ is a $C$-cycle of length $k$ and $B$ is a $C$-cycle of length $n$ with $n \neq k$, then for all $a\in A$ and $b \in B$, the formula $\neg E(a, b)$ holds.

\end{itemize}

Intuitively, $\cM_0$ consists of infinitely many disjoint $C$-cycles of every finite length. Then, for a fixed length $k$ we have an equivalence relation such that there are precisely $k$ classes which contain an element of a $k$-cycle and each of these equivalence classes are compatible with the cycle. 

The following definitions will be important.   
Suppose $A \subseteq \cM_0$. We define the \defn{$E$-closure} of $A$, denoted $\ol{A}$ to be the collection of elements $E$-equivalent to some element of $A$. Let $\iota_A\:A \to \ol{A}$ be the inclusion map.

For $a \in \cM_0$ let $\tau(a) = k$ if $a$ is in a $k$-cycle. Let
\[
\pi_k(A) = \{a \in A\st a \text{ is in a }k\text{-cycle}\}
.
\]

For $E$-equivalence classes $A, B$ let $A \leq B$ if $\min A \leq \min B$ (in the ordering on $\w$). For each $k\in \w$ let $X_k$ be the $\leq$-minimal $E$-equivalence class containing elements in a $k$-cycle. Let $\gamma_k(A) = \ol{A} \cap X_k$, i.e., the set of $X_k$ corresponding to $\ol{A}$. 

Note that the sets $X_k$, the map $A \mapsto \ol{A}$, and the maps $\gamma_k$, $\tau$ and $\pi_k$ are all computable (uniformly in $k$ where applicable). 

If $f\:\ol{A} \to B$ is an embedding then we must have $\tau(f(a)) = \tau(a)$ for all $a \in \ol{A}$. Therefore any such embedding is also an embedding from $\ol{A}$ to $\ol{B}$. 

Further if $f_k\:\pi_k(\ol{A}) \to B$ is an embedding (for $k \in \w$) then there is a unique embedding $\coprod_{i \in \w} f_i\:\ol{A} \to B$ such that
\[\textstyle
(\coprod_{i \in \w} f_i\:\ol{A})\rest[\pi_k(\ol{A})] = f_k
\]
for all $k \in \w$. 

Suppose $\Lang_1$ is a language disjoint from $\Lang_0$. Let $\Lang^* = \Lang_0 \cup \Lang_1$. We say an $\Lang^*$-structure $\cM$ is \defn{compatible with $\cM_0$} if for all $R \in \Lang_{1}$ and $\cM \models R(a_0, \dots, a_{k - 1})$, we have 
\begin{itemize}
\item $\cM \models \bigwedge_{i \leq j \in [k]} E(a_i, a_j)$ and

\item if $\cM \models \bigwedge_{i \in [k]}C(a_i, b_i)$ then $\cM \models R(b_0, \dots, b_{k-1})$. 

\end{itemize}

So $\cM$ is compatible with $\cM_0$ if 
\begin{itemize}
\item we can $\cM$ break up into substructures, one for each $E$-equivalence class, with no relations holding between tuples across equivalence classes, and

\item if two equivalence classes are on cycles of the same length then the corresponding structures are the same on both the equivalence classes (and the cycles maps give rise to an isomorphism between the structures). 
\end{itemize}

In particular, if $(\cN_i)_{i \in \w}$ is a sequence of $\w$-many countable $\Lang_1$-structures where the underlying set of each $\cN_i$ is $X_i$, then there is a unique $\Lang^*$-structure $\coprod_{i \in \w} \cN_i$ which is compatible with $\cM_0$ such that
\[\textstyle
(\coprod_{i \in \w} \cN_i)\rest[X_k] = \cN_k
.
\]

Suppose that  $\cK_i$ is the age of $\cN_i$ for each $i \in \w$. Let $\cK_\w$ be the age of $\coprod_{i \in \w} \cN_i$.  The following are then immediate. 
\begin{itemize}
\item If the sequence $(\cN_i)_{i \in\w}$ is uniformly computable then $\coprod_{i \in \w} \cN_i$ is computable. 

\item  $\cK_i$ has $\HP$ and $\JEP$ for all $i \in \w \cup \{\w\}$.
\end{itemize}

We will now prove two important claims which will allow us to reduce our task to the construction to subages. First though we need some definitions.

\begin{claim}
\label{Lower bound on coAP in finite language: Claim 1}
Suppose $\cK_i$ has $\coAP$ for each $i \in \w$. Then $\cK_\w$ has $\coAP$. 
\end{claim} 
\begin{claimproof}
If $f_B\:\ol{A} \to B$ and $f_C\:\ol{A} \to C$ is such that for all $k \in \w$ there are embeddings $g_{B, k}\:\pi_k(\ol{B}) \to \pi_k(D)$ and $g_{C, k}\:\pi_k(\ol{B}) \to \pi_k(D)$ with $g_{B, k} \circ f_B \rest[\pi_k(\ol{A})] = g_{C, k} \circ f_C \rest[\pi_k(\ol{A})]$ then $g_B\:\ol{B} \to D$ and $g_C\:\ol{B} \to D$ with $g_B \circ f_B = g_C \circ f_C$ when $g_B = \coprod_{k \in \w} g_{B, k}$ and $g_C = \coprod_{k \in \w} g_{C, k}$. 

In particular, this implies that if 
$\compcK_i$ has $\coAP$ 
for all $i \in \w$, 
then $\compcK_\w$ has $\coAP$ as well. 
\end{claimproof}

\begin{claim}
\label{Lower bound on coAP in finite language: Claim 2}
Suppose $(\WitcoAP, \WitFuncoAP)$ is a witness to $\coAP$ for $\cK_\w$. Then, uniformly in $e\in \w$ we can compute a witness to $\coAP$ for $\cK_e$. 
\end{claim}
\begin{claimproof}
Note that as $\Lang^*$ is a relational language, every subset of an $\Lang^*$-structure is a substructure. Further notice that from $(\WitcoAP, \WitFuncoAP)$ we can find a a witness to $\coAP$ such that whenever $\aa$ and $\aa^*$ enumerate the same set (possibly with repetitions), $\aa \in \WitcoAP$ if and only if $\aa^* \in \WitcoAP$. Hence it suffices to identify the structures $\cA$ such that $(\aa, \cA, i) \in \WitcoAP$ for some $\aa$ and $i$. Therefore we will abuse notation and say $\cA \in \WitcoAP$ in this situation. 

For $\cB \in \WitcoAP$ let $\cB_k$ be the collection of elements $b$ in  $\bb$ with $\tau(b) = k$. Let $\cA \in \WitcoAP(\compcK_e)$ if and only if $\cA = \cB_e$ for some $\cB \in \WitcoAP$. 

Suppose $(F_0, F_1)$ is a potential span in $\compcK_e$ with $\dom(F_0) \in \WitcoAP$. Further suppose $\cB \in \WitcoAP$ is such that $\cB_e = \dom(F_0)$. Let $(F_0^*, F_1^*)$ be the maps where 
\begin{itemize}
\item $\dom(F_0)^* = \dom(F_1)^* = \cB_e$, 

\item for $i \in \{0, 1\}$, $\codom(F_i)^* = \codom(F_i) \cup (\cB \setminus \cB_e)$, and

\item for $i \in \{0, 1\}$, $F_i \rest[\dom(F_i)] = F_i$ and $F_i \rest[\cB \setminus \cB_e] = \id$. 

\end{itemize}
It is clear that $(F_0^*, F_1^*)$ is a potential span in $\cK_\w$. Let $\WitFuncoAP(F_0^*, F_1^*) = (G_0, G_1)$. For $i \in \{0, 1\}$ let $G_i^- = G_i \rest[\codom(F_i)]$. We then have 
$\range(G_i^-) \subseteq \pi_e(\range(G_i)$
for $i \in \{0, 1\}$.
Therefore $(G_0^-, G_1^-)$ is an amalgamation diagram over $(F_0, F_1)$. Hence we can let $\WitcoAP(\compcK_e)(F_0, F_1) = (G_0^-, G_1^-)$. 
\end{claimproof}

Note that it is not the case that if 
$\cK_i$ has $\AP$ 
for all $i \in \w$, 
then $\cK_\w$ has $\AP$. This is because the age of $\cM_0$ does not have $\AP$. 

By \cref{Lower bound on coAP in finite language: Claim 1,Lower bound on coAP in finite language: Claim 2}, 
it therefore suffices to construct a sequence of ages $(\cK_e)_{e\in\w}$ each with $\coAP$ such that whenever $(\WitcoAP(\compcK_e), \WitFuncoAP(\compcK_e))$ is a witness to $\coAP$ we can (uniformly in $e$ and the witness) determine whether or not $\{e\}(n)\converges$. 

Let $\Lang_1 = \{B, Y\}$ where both are binary relations. Let $W_{m, n}$ be the structure where 
\begin{itemize}
\item the underlying set consists of elements $\{x_i\}_{i \in \w} \cup \{b^+_{i}\}_{i < m} \cup \{b^-_{i}\}_{i < n}  \cup \{y^+_{i}\}_{i < n} \cup \{y^-_{i}\}_{i < m} \cup \{q_+, q_-\}$; 

\item $B(a,c)$ holds if either 
\begin{itemize}
\item $a = q_\square$ and $c = b^\square_{0}$ for $\square \in \{+, -\}$, or

\item $a = b^\square_{i}$ and $c = b^\square_{i+1}$ for some $i+1 < k_\square$ where $\square \in \{+, -\}$ and $k_+ = m$ and $k_- = n$; 
\end{itemize}

\item $Y(a,c)$ holds if either 
\begin{itemize}
\item $a = q_\square$ and $c = y^\square_{0}$ for $\square \in \{+, -\}$, or 

\item $a = y^\square_{i}$ and $c = y^\square_{i+1}$ for some $i+1 < k_\square$ where $\square \in \{+, -\}$ and $k_+ = n$ and $k_- =m$; and
\end{itemize}

\item no other relations hold.
\end{itemize}

Note the $\{x_i\}_{i \in \w}$ are there simply to ensure the structure is infinite. 

Now we define for $i \in \w$
\[
\cN_i = 
\begin{cases}
    W_{30, 90} & \text{if } \{i\}(0) \converges, \\
    W_{2, 4} & \text{if } \{i\}(0)\diverges.
\end{cases}
\]
Note that if $m_0 \leq m_1$ and $n_0 \leq n_1$ then $W_{m_0, n_0} \subseteq W_{m_1, n_1}$ and so $\{\cN_i\}_{i \in \w}$ are uniformly computable. Therefore $\coprod_{i \in \w} \cN_i$ is a computable structure. 

Working in $\cK_i$ for some $i\in\w$, suppose that $A \in \WitcoAP(\compcK_i)$ with $\{q_+, q_-\}\subseteq A$. Let $B(A), Y(A)$ be the number of elements in $A$ in a $B$-edge and a $Y$-edge respectively. Let $k(A) = \max\{|B(A)|, |Y(A)|\}$.

Suppose $\{i\}(0)\converges$. To get a contradiction also suppose that $k(A) < 14$. Let $A_0$ be the structure which extends $A$ by ensuring that the $B$-chain connected to $q_+$ has length at least $31$ elements. Let $A_1$ be the structure which extends $A$ by ensuring that the $Y$-chain connected to $q_+$ has at least $31$ elements. Note both $A_0, A_1 \in \cK_i$ as there is no way to distinguish $q_+, q_-$ in $\cN_i$ and $\cK_i$ has $\HP$. However, it is impossible to amalgamate $A_0, A_1$ over $\{q_+, q_-\}$. Therefore any element of $\WitcoAP$ containing $\{q_+, q_-\}$ must have at least $14$ elements in $B$-edges and at least $14$ elements in $Y$-edges.  

But, if $\{i\}(0) \diverges$ then there are only $13$ elements in $B$-edges and $13$ elements in $Y$-edges. Therefore, we can computably determine (uniformly in $i$) whether $\{i\}(0)$ halts by looking at an the first element of $\WitcoAP$ containing $\{q_+, q_-\}$. Therefore from any witnesses to the $\CcoAP(\StandardTuringDegree)$ we can compute $\TuringJump$. 

Finally, we need to show that $\compcK_i$ has $\coAP$. But this follows from the fact that any substructure of $\cN_i$  containing the maximal $W_{m, n}$ is an amalgamation base. 
\end{proof}

%% file: ingredients/LowerBoundCoAPInfiniteProof.tex
\begin{proof}
Let $(u_i)_{i \in \w}$ be a non-decreasing computable enumeration of finite sets such that $\bigcup_{i \in \w} u_i = \{e \in \w \st \{e\}(0) \converges\}$. Let $f\:\w \times \w \to \w$ be the computable function where to compute $f(e, n)$ we run the following algorithm and return the $n$th output. Call this ``Program $f$''.\nl\nl
\ul{Stage $0$}: \nl
Create a variable which takes values in $\w$ and which we think of as ``maximal oracle call'' made. We denote this variable by $oc$. We set $oc$ to $0$. 
\nl\nl
\ul{Stage $n+1$}:
We break the stage into three cases. \nl
\ul{Case $1$:} $u_{n+1}\rest[oc] = u_n\rest[oc]$ and the Turing machine simulating $\{e\}^{u_n}(0)$ has halted after $n$ steps. \nl
In this case output $1$, i.e., $f(e, n) = 1$. \nl\nl
\ul{Case $2$:} $u_{n+1}\rest[oc] = u_n\rest[oc]$ and $\{e\}^{u_{n+1}}(0)$ has not halted after $n$ steps. \nl
Run $\{e\}^{u_{n+1}}(0)$ for the $(n+1)$st step.  If an oracle call larger than $oc$ was made, then update the value of $oc$ to be the index of this oracle call. Output $0$, i.e., $f(e, n) = 0$. \nl
\nl
\ul{Case $3$:} $u_{n+1}\rest[oc] \neq u_n\rest[oc]$. \nl
Run $\{e\}^{u_{n+1}}(0)$ for $n+1$ steps and set $oc$ to be the largest oracle call made. Output $0$, i.e., $f(e, n) = 0$.  
\begin{claim}
\label{Lower bound on coAP in infinite language: Claim 1}
For any $e \in \w$ the following are equivalent 
\begin{itemize}

\item[(a)] $\{n \st f(e, n) = 0\}$ is finite, 

\item[(b)] $\{e\}^{\TuringJump}(0) \converges$

\end{itemize}
\end{claim}
\begin{claimproof}
Suppose (b) holds, i.e., $\{e\}^{\TuringJump}(0)\converges$. Then the program only makes a finite number of oracle calls. Let $\ell$ be the largest such oracle call. Let $k$ be such that 
\begin{itemize}
\item $\ell < k$, i.e., every oracle call in the computation $\{e\}^{\TuringJump}(0)$ is less than $k$, and

\item for all $k^* \geq k$, we have $u_{k^*}\cap [\ell] = u_k \cap [\ell]$.
\end{itemize}
Note we can always find such a $k$ as the $U_i$'s are non-decreasing. But then for all $k^* > k$ we have $f(e, k^*) = 1$ and so (a) holds. 

Now suppose (b) does not hold, i.e., $\{e\}^{\TuringJump}(0) \diverges$. Then either $\{e\}^{\TuringJump}(0)$ makes arbitrarily large oracle calls,  or there is a bound on the size of the oracle calls and the program runs forever. In the former situation, Case $3$ occurs infinitely often, and in the latter situation, Case $2$ occurs infinitely often. In either situation, $\{n \st f(e, n) = 0\}$ is infinite.
\end{claimproof}

We will construct a computable age $\compcK$ with $\coAP$ but where, for any witness to $\coAP$, we can compute whether or not $\{n \st f(e, n) = 0\}$ is finite for each $e$. By \cref{Lower bound on coAP in infinite language: Claim 1} this will suffice to prove our result. For convenience we will let $\Phi_e = \{n \st f(e, n) = 0\}$.

Let $\Lang^- = \{P_i\}_{i \in \w}$ be a computable language consisting of only unary relations. Let $\cN^-$ be a computable $\Lang^-$-structure such that $\{P_i^{\cN^-}\}_{i \in \w}$ is a partition of $\cN^-$ and each $P_i^{\cN^-}$ is infinite. 

Suppose 
\begin{itemize}

\item $(\Lang_i)_{i \in \w}$ is a uniformly computable sequence of disjoint relational languages each disjoint from $\Lang^-$, 

\item $\Lang = \bigcup_{i \in \w} \Lang_i \cup \Lang^-$,  

\item 
$\cN_i$ is an $\Lang_i$-structure with underlying set $P^{\cN^-}$
for $i \in \w$, and

\item $\cN_\w$ is the $\Lang$-structure where 
\begin{itemize}
\item $\cN_\w \rest[\Lang^-] = \cN^-$, 

\item if $\cN_\w \models R(a_0, \dots, a_{n-1})$ where $R$ is an atomic formula in $\Lang_k$, then $\cN_\w \models \bigwedge_{i \in [n]} P_k(a_i)$, and

\item $(\cN_\w\rest[P_i^{\cN_\w}]) \rest[\Lang_i] = \cN_i$. 
\end{itemize}
\end{itemize}
For each $i \in \w \cup \{\w\}$, let $\compcK_i$ be the canonical computable age of $\cN_i$.

The following are then immediate. 
\begin{itemize}

\item[(a)]  $\compcK_i$ has $\HP$ for all $i \in \w \cup \{\w\}$. 

\item[(b)]  $\compcK_i$ has $\JEP$ for all $i \in \w \cup \{\w\}$.

\item[(c)] If  $\compcK_i$ has $\coAP$  for all $i \in \w$,  then $\compcK_\w$ has $\coAP$.  

\item[(d)] From any $\StandardTuringDegree$-computable witness for $\coAP$ in $\cK_\w$, we can compute an $\StandardTuringDegree$-computable witness for $\coAP$ in $\cK_n$ for all $n \in \w$, uniformly in $n$.
\end{itemize}    

It therefore suffices to define, uniformly in $n$, a computable structure $\cN_n$ such that $\compcK_n$ has $\coAP$ and such that we can determine uniformly in $n$ whether or not $\{n \st f(e, n) = 0\}$ is infinite from a witness to $\coAP$ in $\compcK_n$. 
 
We now define $\compcK_e$, uniformly in $e$.  Let $\Lang_e = \{B^e, Y^e, R^e\} \cup \{Q^e\}\cup \{U^e_i\}_{i \in \w}$ where $B^e, Y^e, R^e$ are binary relation symbols, $Q^e$ is a unary relation symbol and $U_i^e$ is a unary relation symbol for each $i \in \w$. In what follows, for ease of reading, we will omit the superscripts as they will always be $e$. 

Let $m_e = |\Phi_e|$ if $\Phi_e$ is finite and $\w$ otherwise. For each $\sigma \in 2^{<\w}$ we define the structure $W_{\sigma}$ as follows.

\begin{itemize}
\item The underlying set is
\[\{b^{+, \sigma}_{i}\}_{i < 1+m_e} \cup \{b^{-, \sigma}_{i}\}_{i < m_e}  \cup \{y^{+, \sigma}_{i}\}_{i< m_e} \cup \{y^{-, \sigma}_{i}\}_{i < 1+m_e} \cup \{q_+^{\sigma}, q_-^{\sigma}\}.
\]
 
\item $Q(a, b)$ holds if and only if $\{a, b\} =  \{q_+^{\sigma}, q_-^{\sigma}\}$. We call $q_+^{\sigma}, q_-^{\sigma}$ the \defn{roots} of $W_\sigma$. 

\item $B(a,c)$ holds if either 
\begin{itemize}
\item $a = q_\square^\sigma$ and $c = b^{\square, \sigma}_{0}$ for $\square \in \{+, -\}$, or

\item $a = b^{\square, \sigma}_{i}$ and $c = b^{\square, \sigma}_{i+1}$ for some $i +1 < k_\square$ where $\square \in \{+, -\}$ and $k_+ = 1+m_e$ and $k_-= m_e$. 

\end{itemize}

\item $Y(a,c)$ holds if either 
\begin{itemize}
\item $a = q_\square^\sigma$ and $c = y^{\square, \sigma}_{0}$ for $\square \in \{+, -\}$,  or

\item $a = y^{\square, \sigma}_{i}$ and $c \in y^{\square, \sigma}_{i+1}$ for some $i < k_\square$ where $\square \in \{+, -\}$ and $k_+ = m_e$ and $k_-= 1+m_e$. 

\end{itemize}

\item For $i \in \w$, if $f(e, i) = 1$ then $W_{\sigma} \models (\forall x)\, \neg U_i(x)$. 
  
\item For $i \in \w$, if $f(e, i) = 0$ and $k = |\{j < i \st j \in \Phi_e\}|$ then 
\begin{itemize}
\item if $k < \len(\sigma)$ and  $\sigma(k) = 1$ and $x \not \in \{q_+, q_-\}$, then $W_{\sigma} \models U_i(x)$,  and

\item otherwise, $W_\sigma \models \neg U_i(x)$. 
\end{itemize}

\end{itemize}

We let $\cN_e$ be the disjoint amalgamation of the structures $\{W_\sigma\st \sigma \in \{0, 1\}^{<\w}\}$. Note that the structures $W_\sigma$ are uniformly c.e. in $\sigma$ and so $\cN_e$ is computable. 

One can think of the $\{B, Y\}$-structure of $W_\sigma$ as consisting of two roots $q_+^\sigma, q_-^\sigma$ and attached to each root is a $B$-chain and a $Y$-chain, where the $B$-chain attached to $q_+^\sigma$ is longer than the $Y$-chain and the reverse is true of the chains attached to root $q_-^\sigma$. On top of the $\{B, Y\}$ structure we also add a $\{U_i\}_{i \in \w}$ structure where each such unary relation either holds or doesn't hold of all non-root elements. 

In order to determine whether or not $U_i$ holds we look at the function $f(e, i)$. On the $i$ for which $f(e, i) = 1$, the relation $U_i$ does not hold for any $\sigma$. However, if $f(e, i) = 0$ then we look at the number of $j$ less than $i$ such that $f(e, j) = 0$, and we use $\sigma$ applied to that number to determine whether or not $U_i$ holds. 

For $\sigma_0, \sigma_1 \in 2^{<\w}$, let $\iota_{\sigma_0, \sigma_1}\:W_{\sigma_0} \to W_{\sigma_1}$ be the map such that for $x^{\sigma_0} \in W_{\sigma_0}$ we have $\iota_{\sigma_0, \sigma_1}(x^{\sigma_0}) = x^{\sigma_1}$ (where $x$ is either $q_+$ or $q_-$, or is $b_i^+$, $b_i^-$, $y_i^+$, or $y_i^-$ for some $i$). 

The effect of this is that if $\Phi_e$ is finite, and if $\len(\sigma) \geq 1 + \max \Phi_e$, then the map $\iota_{\sigma\tconcat0, \sigma\tconcat1}$ from $W_{\sigma\tconcat0}$ to $W_{\sigma\tconcat1}$ is an isomorphism.

But, if $\Phi_e$ is infinite, then for every $\sigma\in 2^{<\w}$ and every $V \subseteq W_{\sigma\tconcat0}$ with $V\not \subseteq \{q_+^{\sigma\tconcat0}, q_-^{\sigma\tconcat0}\}$ the map $\iota_{\sigma\tconcat0, \sigma\tconcat1}\rest[V]$ from $V$ to $\iota_{\sigma\tconcat0, \sigma\tconcat1}``[V] \subseteq W_{\sigma\tconcat1}$ is not an embedding. This is because whenever
\[
|\{j < i \st f(e, j) = 0\}| = \len(\sigma)
\]
and $f(e, i) = 0$ then
\[
W_{\sigma\tconcat0} \models (\forall x)\, \neg U_i(x)
\]
and 
\[
W_{\sigma\tconcat1} \models (\forall x)\ \bigl(x \in \{q_+^{\sigma\tconcat1}, q_-^{\sigma\tconcat1}\} \Or U_i(x)\bigr)
.
\]

Therefore, if we can find an $\ell \in \w$ that guarantees either (i) $\ell$ is  at least $1 + \max \Phi_e$ and $\Phi_e$ is finite, or (ii) $\Phi_e$ is infinite, then we can determine from $\EmbedInfo(\compcK_e)$ which of (i) or (ii) holds and hence whether or not $\Phi_e$ is finite. 

\begin{claim}
\label{Lower bound on coAP in infinite language: Claim 2}
Suppose $(\WitcoAP, \WitFuncoAP)$ is a witness to $\coAP$ for $\cN_e$. Then, uniformly in $(\WitcoAP, \WitFuncoAP)$ and $e$, we can find a number $g_e$ such that whenever $\Phi_e$ is finite, $g_e = 1+|\Phi_e|$. 
\end{claim}
\begin{claimproof}
Let $i$ be such that $\AgeStr(i)$ has two elements $a, b$ and $\AgeStr(i) \models Q(a, b)$.  Uniformly in $\WitcoAP$ we can find an $F \in \WitcoAP$ with $\dom(F) = \AgeIndex(i)$. Let $\cC$ be the underlying structure of $\codom(F)$. Let $\sigma \in 2^{<\w}$ be such that $\range(F) \subseteq W_\sigma$.

Let $\cD = \cC \setminus W_\sigma$, i.e., the collection of elements in $\cC$ not in $W_\sigma$. Let $\cE = \cC \cap W_\sigma$. 

Let $\ell^+_B$, $\ell^-_B$ be the length of the $B$-chain in $\cC$ connected to $q_+^\sigma$, $q_-^\sigma$ respectively.  Let $\ell^+_Y$, $\ell^-_Y$ be the length of the $Y$-chain in $\cC$ connected to $q_+^\sigma$, $q_-^\sigma$ respectively.  Let $g_e = \max \{\ell^+_B, \ell^-_B, \ell^+_Y, \ell^-_Y\}$. Note $g_e$ is then at most $1 + |\Phi_e|$. 

Now suppose $\Phi_e$ is finite. Then $W_{\tau_0} \cong W_{\tau_1}$ provided that $\tau_0 \rest[|\Phi_e|+1] = \tau_1 \rest[|\Phi_e|+1]$. 

Suppose, towards a contradiction, that $g_e < 1+ |\Phi_e|$. In this case there is a unique embedding $\beta\:\cE \to W_\sigma$ with $\beta(q_-^\sigma) = q_+^\sigma$ and $\beta(q_+^\sigma) = q_-^\sigma$. Let $\gamma_0, \gamma_1\:\cC \to \cC \cup W_\sigma$ with $\gamma_0 = \id\rest[\cC]$ and $\gamma_1 = \beta \cup \id\rest[\cD]$. As $F \in \WitcoAP$ there must be $\eta_0, \eta_1\:\cC \cup W_\sigma \to \cG$ for some $\cG$ where $\eta_0 \circ \gamma_0 = \eta_1 \circ \gamma_1$.

But the maximal $B$-chain attached to $q_+^\sigma$ has length $1 + | \Phi_e|$ in $W_\sigma$ and hence also in $\cC \cup W_\sigma$. Therefore the maximal $B$-chain in $\cG$ attached $\eta_0 \circ \gamma_0(q_+^\sigma)$ has length (at least) $1 + |\Phi_e|$. But the maximal $Y$-chain in $\cG$ attached to $\gamma_1(q_+^\sigma)$ has length  $1 + |\Phi_e|$ and so the maximal $Y$-chain attached to $\eta_1 \circ \gamma_1(q_+^\sigma)$ has length (at least) $1 + |\Phi_e|$. But as $\eta_0 \circ \gamma_0(q_+^\sigma) =\eta_1 \circ \gamma_(q_+^\sigma)$ this implies there is an element in $\cG$ which is the start of both a $B$-chain and a $Y$-chain, each of length (at least) $1 + |\Phi_e|$, which contradicts how $\cN_e$ was constructed. 

Therefore if $\Phi_e$ is finite we must have $g_e = 1 + |\Phi_e|$. 
\end{claimproof}

We need one more ingredient to complete
the result. 

\begin{claim}
\label{Lower bound on coAP in infinite language: Claim 3}
$\compcK_e$ has $\coAP$. 
\end{claim}
\begin{claimproof}
Suppose $\AgeIndex(i) \in \compcK_e$. We need to show that there is a map $F$ with $\dom(F) = \AgeIndex(i)$ and $\codom(F)$ an amalgamation base. We break into two cases depending on whether or not $|\Phi_e|$ is finite. \nl\nl
\ul{Case 1:} $\Phi_e$ is finite. \nl
Because $\AgeIndex(i) \in \compcK_e$ and $\compcK_e$ is the canonical computable age of $\cN_e$, there must be a finite number of $\sigma_0, \dots, \sigma_{k-1}$ such that $\AgeStr(i) \subseteq \bigcup_{i \in [k]} W_\sigma$. Let $\cB$ be this union and $F$ be the inclusion map. 

Suppose 
that $\alpha_i\:\cB \to \cC_i$,
for $i \in \{0, 1\}$, 
are embeddings with $\cC_i \in \compcK_e$. Then no relations hold between any element of $\cC_i \setminus \alpha_i``[\cB]$ and any element of $\alpha_i``[\cB]$. Further, because $W_{\tau_0} \cong W_{\tau_1}$ whenever $\tau_0\rest[1+ |\Phi_e|] = \tau_1\rest[1+ |\Phi_e|]$ we can find embeddings of $\cC_i \setminus \alpha_i``[\cB]$ into $\cN_e$ whose images are disjoint with each other and with the image of $\cB$. 

Therefore, we can find an amalgamation of $\alpha_0$ and $\alpha_1$, and so $\cB$ is an amalgamation base. As $\AgeIndex(i)$ was arbitrary, this implies that $\compcK_e$ has $\coAP$. \nl\nl
\ul{Case 2:} $\Phi_e$ is infinite. \nl
Suppose $\cC \subseteq \cN_e$. We say $c \in \cC$ is \defn{closed} in $\cC$ if whenever $c$ is in a $B$-chain or $Y$-chain connected to a root in $\cN_e$ then $c$ is in a $B$-chain or $Y$-chain connected to a root in $\cC$. We say $\cC$ is \defn{closed} if 
\begin{itemize}
\item  
$c$ is closed in $\cC$
for every $c \in \cC$, 
i.e., every non-root element is connected to a root, 

\item every root is connected via some $B$-edge or $Y$-edge to a non-root, and 

\item if $q \in \cC$ and $\cN_e \models Q(q, q')$ then $q' \in \cC$, i.e., if we have a root we also have its pair. 
\end{itemize}   

If $a \in \cN_e$ is a non-root, then because $\Phi_e$ is infinite, we can read off from the $\{U_i\}_{i \in \w}$-structure of $a$ which $W_\sigma$ it is in. Similarly, if $\cC$ is closed we can read off from each root which $W_\sigma$ it came from. As such, if $\cC_0, \cC_1$ are closed there at most one embedding from $\cC_0$ to $\cC_1$.

Further, every finite subset of $\cN_e$ is contained in a finite closed subset of $\cN_e$. 

Suppose $\cB$ is closed and 
each $\alpha_i\:\cB \to \cC_i$, 
for $i \in \{0, 1\}$,  
is an embedding with $\cC_i \in \compcK_e$. Let $\cC_i^+$ be a closed finite subset of $\cN_e$ containing $\cC_i$. There is then a unique map $\alpha_i^+\:\cB \to \cC_i^+$. Further, as there is a unique embedding from $\cB$ into $\cN_e$ and a unique embedding from $\cC_i^+$ into $\cN_e$, we can assume without loss of generality that $\alpha_i^+$ is an inclusion. But then we can let $\cD \subseteq \cN_e$ be any closed subset containing $\cC_0^+$ and $\cC_1^+$. Then for
each $i\in \{0, 1\}$,
there is a unique map from $\cC_i^+$ to $\cD$.
Therefore $\cD$ is an amalgamation of $\alpha_0$ and $\alpha_1$. 

In particular, this implies that $\cB$ is an amalgamation base. Hence we can let $\codom(F)$ be any element of $\compcK_e$ whose underlying structure is $\cB$. Therefore $\compcK_e$ has $\coAP$. 
\end{claimproof}

Let $\cN$ be the disjoint amalgamation of 
$\cN_\w$ and
the structure from \cref{Lower bound on coAP in finite language}.
Let $\compcK$ be the age of $\cN$.
Then by \cref{Lower bound on coAP in finite language} and \cref{Lower bound on coAP in infinite language: Claim 3}, the age $\compcK$ has $\coAP$. Now suppose $\compcK$ has $\CcoAP(\StandardTuringDegree)$ for some Turing degree $\StandardTuringDegree$. By \cref{Lower bound on coAP in finite language} we can compute $\TuringJump$ from $\StandardTuringDegree$, and so by \cref{Embdedding info is computable from 0'} we can compute $\EmbedInfo(\compcK)$ from $\StandardTuringDegree$ as well. But then by \cref{Lower bound on coAP in infinite language: Claim 2}, uniformly in $e$ we can determine from $\StandardTuringDegree$ whether or not $\Phi_e$ is finite. Hence $\TuringDoubleJump \Turingleq \StandardTuringDegree$. 
\end{proof}

%% file: bibliography.bib
@article{futureworkAFM,
  author    = {Nathanael Ackerman and Cameron Freer and Mostafa Mirabi},
  title     = {Computable Weak {Fra\"iss\'e} Limits},
  note = {In preparation},
  fyear      = {2025},
  fnote      = {To appear},
}

@article{MR1224221,
    AUTHOR = {Droste, Manfred and G\"{o}bel, R\"{u}diger},
     TITLE = {Universal domains and the amalgamation property},
   JOURNAL = {Math. Structures Comput. Sci.},
  FJOURNAL = {Mathematical Structures in Computer Science. A Journal in the
              Applications of Categorical, Algebraic and Geometric Methods
              in Computer Science},
    VOLUME = {3},
      YEAR = {1993},
    NUMBER = {2},
     PAGES = {137--159},
      ISSN = {0960-1295,1469-8072},
   MRCLASS = {68Q55 (03C99 18B99)},
  MRNUMBER = {1224221},
MRREVIEWER = {David\ B.\ Benson},
       DOI = {10.1017/S0960129500000177},
       URL = {https://doi.org/10.1017/S0960129500000177},
}

@inproceedings{DBLP:conf/lics/DrosteG90,
  author       = {Manfred Droste and
                  R{\"{u}}diger G{\"{o}}bel},
  title        = {Universal Domains in the Theory of Denotational Semantics of Programming
                  Languages},
  booktitle    = {Proceedings of the Fifth Annual Symposium on Logic in Computer Science {(LICS} '90), Philadelphia, PA, June 4-7, 1990},
  fbooktitle    = {Proceedings of the Fifth Annual Symposium on Logic in Computer Science {(LICS} '90), Philadelphia, Pennsylvania, USA, June 4-7, 1990},
  pages        = {19--34},
  publisher    = {{IEEE} Computer Society},
  year         = {1990},
  url          = {https://doi.org/10.1109/LICS.1990.113730},
  doi          = {10.1109/LICS.1990.113730},
  bibsource    = {dblp computer science bibliography, https://dblp.org}
}

@article{MR3689377,
    AUTHOR = {Bodirsky, Manuel and Jonsson, Peter and Van Pham, Trung},
     TITLE = {The complexity of phylogeny constraint satisfaction problems},
   JOURNAL = {ACM Trans. Comput. Log.},
  FJOURNAL = {ACM Transactions on Computational Logic},
    VOLUME = {18},
      YEAR = {2017},
    NUMBER = {3},
     PAGES = {Art. 23, 42pp.},
      ISSN = {1529-3785,1557-945X},
   MRCLASS = {68R05 (03C05 05C90 68Q17 68Q25 90C60 92D15)},
  MRNUMBER = {3689377},
       DOI = {10.1145/3105907},
       URL = {https://doi.org/10.1145/3105907},
}

@article{MR2075214,
    AUTHOR = {Cristani, M. and Hirsch, R.},
     TITLE = {The complexity of constraint satisfaction problems for small
              relation algebras},
   JOURNAL = {Artificial Intelligence},
  FJOURNAL = {Artificial Intelligence},
    VOLUME = {156},
      YEAR = {2004},
    NUMBER = {2},
     PAGES = {177--196},
      ISSN = {0004-3702,1872-7921},
   MRCLASS = {68Q25 (03G15 68Q17 68T20)},
  MRNUMBER = {2075214},
       DOI = {10.1016/j.artint.2004.02.003},
       URL = {https://doi.org/10.1016/j.artint.2004.02.003},
}

@article{MR3810321,
    AUTHOR = {Badia, Guillermo and Noguera, Carles},
     TITLE = {Fra\"{\i}ss\'{e} classes of graded relational structures},
   JOURNAL = {Theoret. Comput. Sci.},
  FJOURNAL = {Theoretical Computer Science},
    VOLUME = {737},
      YEAR = {2018},
     PAGES = {81--90},
      ISSN = {0304-3975,1879-2294},
   MRCLASS = {03E72 (03C99)},
  MRNUMBER = {3810321},
       DOI = {10.1016/j.tcs.2018.05.010},
       URL = {https://doi.org/10.1016/j.tcs.2018.05.010},
}

@inproceedings{DBLP:conf/pods/BojanczykST13,
author = {Boja\'{n}czyk, Miko\l{}aj and Segoufin, Luc and Toru\'{n}czyk, Szymon},
  editor       = {Richard Hull and
                  Wenfei Fan},
  title        = {Verification of database-driven systems via amalgamation},
  booktitle    = {Proceedings of the 32nd {ACM} {SIGMOD-SIGACT-SIGART} Symposium on
                  Principles of Database Systems, {PODS} 2013, New York, NY, {USA} -
                  June 22 - 27, 2013},
  pages        = {63--74},
  publisher    = {{ACM}},
  year         = {2013},
  url          = {https://doi.org/10.1145/2463664.2465228},
  doi          = {10.1145/2463664.2465228},
  bibsource    = {dblp computer science bibliography, https://dblp.org}
}

@inproceedings{DBLP:conf/lics/KhoussainovNRS04,
  author       = {Bakhadyr Khoussainov and
                  Andr{\'{e}} Nies and
                  Sasha Rubin and
                  Frank Stephan},
  title        = {Automatic Structures: Richness and Limitations},
  booktitle    = {Proceedings of the 19th {IEEE} Symposium on Logic in Computer Science {(LICS 2004)}, 14-17 July 2004, Turku, Finland},
  fbooktitle    = {19th {IEEE} Symposium on Logic in Computer Science {(LICS 2004)}, 14-17 July 2004, Turku, Finland, Proceedings},
  pages        = {44--53},
  publisher    = {{IEEE} Computer Society},
  year         = {2004},
  url          = {https://doi.org/10.1109/LICS.2004.1319599},
  doi          = {10.1109/LICS.2004.1319599},
  bibsource    = {dblp computer science bibliography, https://dblp.org}
}

@article{MR3016251,
    AUTHOR = {Fouch\'{e}, Willem L.},
     TITLE = {Martin-{L}\"{o}f randomness, invariant measures and countable
              homogeneous structures},
   JOURNAL = {Theory Comput. Syst.},
  FJOURNAL = {Theory of Computing Systems},
    VOLUME = {52},
      YEAR = {2013},
    NUMBER = {1},
     PAGES = {65--79},
      ISSN = {1432-4350,1433-0490},
   MRCLASS = {03D32 (03C13 03E15 05C55 37B05)},
  MRNUMBER = {3016251},
MRREVIEWER = {Gabriel\ Debs},
       DOI = {10.1007/s00224-012-9419-y},
       URL = {https://doi.org/10.1007/s00224-012-9419-y},
}

@incollection{CenzerAdamsNg,
author = { Douglas   Cenzer  and  Francis   Adams  and  Keng Meng   Ng },
title = {Computability and categoricity of weakly homogeneous {Boolean} algebras and $p$-groups},
booktitle = {Aspects of Computation and Automata Theory with Applications},
    SERIES = {Lecture Notes Series. Institute for Mathematical Sciences. National University of Singapore},
fchapter = {},
pages = {141-158},
doi = {10.1142/9789811278631_0006},
    EDITOR = {Greenberg, N. and Jain, S. and Ng, K. M. and Schewe, S. and Stephan, F. and Yang, Y.},
    fEDITOR = {Greenberg, Noam and Jain, Sanjay and Ng, Keng Meng and Schewe, Sven and Stephan, Frank and Yang, Yue},
 PUBLISHER = {World Scientific Publishing Co. Pte. Ltd., Hackensack, NJ},
YEAR = {2024},
URL = {https://www.worldscientific.com/doi/abs/10.1142/9789811278631_0006},
feprint = {https://www.worldscientific.com/doi/pdf/10.1142/9789811278631_0006},
}

@article{MR3722988,
    AUTHOR = {Adams, Francis and Cenzer, Douglas},
     TITLE = {Computability and categoricity of weakly ultrahomogeneous
              structures},
   JOURNAL = {Computability},
  FJOURNAL = {Computability. The Journal of the Association CiE},
    VOLUME = {6},
      YEAR = {2017},
    NUMBER = {4},
     PAGES = {365--389},
      ISSN = {2211-3568,2211-3576},
   MRCLASS = {03C57 (03C35 03C50)},
  MRNUMBER = {3722988},
MRREVIEWER = {Alexandre\ Ivanov},
       DOI = {10.3233/com-170070},
       URL = {https://doi.org/10.3233/com-170070},
}

@phdthesis{MR3034686,
    AUTHOR = {Steiner, Rebecca M.},
     TITLE = {Reducibility, {D}egree {S}pectra, and {L}owness in {A}lgebraic
              {S}tructures},
      fNOTE = {Thesis (Ph.D.)--City University of New York},
      SCHOOL = {City University of New York},
fPUBLISHER = {ProQuest LLC, Ann Arbor, MI},
      YEAR = {2012},
     PAGES = {118},
      ISBN = {978-1267-34693-3},
   MRCLASS = {99-05},
  MRNUMBER = {3034686},
url = {https://www.proquest.com/docview/1018745163},
       fURL =
              {http://gateway.proquest.com/openurl?url_ver=Z39.88-2004&rft_val_fmt=info:ofi/fmt:kev:mtx:dissertation&res_dat=xri:pqm&rft_dat=xri:pqdiss:3508862},
}

@article{MR3944680,
    AUTHOR = {Fokina, Ekaterina and Harizanov, Valentina and Turetsky,
              Daniel},
     TITLE = {Computability-theoretic categoricity and {S}cott families},
   JOURNAL = {Ann. Pure Appl. Logic},
  FJOURNAL = {Annals of Pure and Applied Logic},
    VOLUME = {170},
      YEAR = {2019},
    NUMBER = {6},
     PAGES = {699--717},
      ISSN = {0168-0072,1873-2461},
   MRCLASS = {03D45 (03C57 03D80)},
  MRNUMBER = {3944680},
MRREVIEWER = {Ivan\ V.\ Latkin},
       DOI = {10.1016/j.apal.2019.01.003},
       URL = {https://doi.org/10.1016/j.apal.2019.01.003},
}

@article{MR2140630,
    AUTHOR = {Kechris, A. S. and Pestov, V. G. and Todorcevic, S.},
     TITLE = {Fra\"{\i}ss\'{e} limits, {R}amsey theory, and topological
              dynamics of automorphism groups},
   JOURNAL = {Geom. Funct. Anal.},
  FJOURNAL = {Geometric and Functional Analysis},
    VOLUME = {15},
      YEAR = {2005},
    NUMBER = {1},
     PAGES = {106--189},
      ISSN = {1016-443X,1420-8970},
   MRCLASS = {37B05 (03C15 03E02 03E15 05D10 22F50 43A07 54H20)},
  MRNUMBER = {2140630},
MRREVIEWER = {Eli\ Glasner},
       DOI = {10.1007/s00039-005-0503-1},
       URL = {https://doi.org/10.1007/s00039-005-0503-1},
}

@article{MR3274785,
    AUTHOR = {Angel, Omer and Kechris, Alexander S. and Lyons, Russell},
     TITLE = {Random orderings and unique ergodicity of automorphism groups},
   JOURNAL = {J. Eur. Math. Soc. (JEMS)},
  FJOURNAL = {Journal of the European Mathematical Society (JEMS)},
    VOLUME = {16},
      YEAR = {2014},
    NUMBER = {10},
     PAGES = {2059--2095},
      ISSN = {1435-9855,1435-9863},
   MRCLASS = {03C15 (03C98 22F50 37A25)},
  MRNUMBER = {3274785},
MRREVIEWER = {Yonatan\ Gutman},
       DOI = {10.4171/JEMS/483},
       URL = {https://doi.org/10.4171/JEMS/483},
}

@article{MR0232725,
    AUTHOR = {Calais, {Jean-Pierre}},
     TITLE = {Relation et multirelation pseudo-homog\`enes},
   JOURNAL = {C. R. Acad. Sci. Paris S\'{e}r. A-B},
  FJOURNAL = {Comptes Rendus Hebdomadaires des S\'{e}ances de l'Acad\'{e}mie
              des Sciences. S\'{e}ries A et B},
    VOLUME = {265},
      YEAR = {1967},
     PAGES = {A2--A4},
      ISSN = {0151-0509},
   MRCLASS = {08.30},
  MRNUMBER = {232725},
MRREVIEWER = {F.\ M.\ Sioson},
}

@article{MR0233739,
    AUTHOR = {{Calais}, {Jean-Pierre}},
     TITLE = {Relations pseudo-homog\`enes, application \`a la th\'{e}orie
              des arbres},
   JOURNAL = {C. R. Acad. Sci. Paris S\'{e}r. A-B},
  FJOURNAL = {Comptes Rendus Hebdomadaires des S\'{e}ances de l'Acad\'{e}mie
              des Sciences. S\'{e}ries A et B},
    VOLUME = {266},
      YEAR = {1968},
     PAGES = {A324--A326},
      ISSN = {0151-0509},
   MRCLASS = {06.20 (02.00)},
  MRNUMBER = {233739},
MRREVIEWER = {J.\ R.\ B\"{u}chi},
}

@article{MR1162490,
    AUTHOR = {Truss, J. K.},
     TITLE = {Generic automorphisms of homogeneous structures},
   JOURNAL = {Proc. London Math. Soc. (3)},
  FJOURNAL = {Proceedings of the London Mathematical Society. Third Series},
    VOLUME = {65},
      YEAR = {1992},
    NUMBER = {1},
     PAGES = {121--141},
      ISSN = {0024-6115,1460-244X},
   MRCLASS = {20B27 (20B07)},
  MRNUMBER = {1162490},
MRREVIEWER = {Michael\ Klemm},
       DOI = {10.1112/plms/s3-65.1.121},
       URL = {https://doi.org/10.1112/plms/s3-65.1.121},
}

@phdthesis{MR3697592,
    AUTHOR = {Kruckman, Alex},
     TITLE = {Infinitary {L}imits of {F}inite {S}tructures},
      fNOTE = {Thesis (Ph.D.)--University of California, Berkeley},
      SCHOOL = {University of California, Berkeley},
 fPUBLISHER = {ProQuest LLC, Ann Arbor, MI},
      YEAR = {2016},
     PAGES = {157},
      ISBN = {978-1369-84130-5},
   MRCLASS = {99-05},
  MRNUMBER = {3697592},
       URL ={https://www.proquest.com/docview/1917740694},
       fURL = {http://gateway.proquest.com/openurl?url_ver=Z39.88-2004&rft_val_fmt=info:ofi/fmt:kev:mtx:dissertation&res_dat=xri:pqm&rft_dat=xri:pqdiss:10189549},
}

@misc{KruckmanNotes,
    author = {Alex Kruckman},
    title = {Notes on generalized {Fra\"{i}ss\'{e}} theory},
    fhowpublished = {},
    faddress = {},
    year = {2015},
url = {https://math.berkeley.edu/~kruckman/fraisse.pdf} 
}

@article{MR4433330,
    AUTHOR = {Krawczyk, Adam and Kruckman, Alex and Kubi\'{s}, Wies{\l}aw and
              Panagiotopoulos, Aristotelis},
     TITLE = {Examples of weak amalgamation classes},
   JOURNAL = {MLQ Math. Log. Q.},
  FJOURNAL = {MLQ. Mathematical Logic Quarterly},
    VOLUME = {68},
      YEAR = {2022},
    NUMBER = {2},
     PAGES = {178--188},
      ISSN = {0942-5616,1521-3870},
   MRCLASS = {03C07 (03C50)},
  MRNUMBER = {4433330},
MRREVIEWER = {Vera\ Koponen},
DOI = {10.1002/malq.202100037}
}

@article{MR4292067,
    AUTHOR = {Krawczyk, Adam and Kubi\'{s}, Wies{\l}aw},
     TITLE = {Games with finitely generated structures},
   JOURNAL = {Ann. Pure Appl. Logic},
  FJOURNAL = {Annals of Pure and Applied Logic},
    VOLUME = {172},
      YEAR = {2021},
    NUMBER = {10},
     PAGES = {103016, 13pp.},
     fPAGES = {Paper No. 103016, 13},
      ISSN = {0168-0072,1873-2461},
   MRCLASS = {03C07 (03C50)},
  MRNUMBER = {4292067},
MRREVIEWER = {Santi\ Spadaro},
       DOI = {10.1016/j.apal.2021.103016},
       URL = {https://doi.org/10.1016/j.apal.2021.103016},
}

@misc{deRancourtSlides,
    author = {de Rancourt, No\'{e}},
    title = {Weak {Fra\"{i}ss\'{e}} classes and $\aleph_0$-categoricity},
    howpublished = {Workshop on Generic Structures},
    address = {Prague},
    year = {2021},
    fnote = {Slides},
url = {https://www.karlin.mff.cuni.cz/~rancourt/Conf/Generic2021.pdf}
}

@article{MR4369354,
    AUTHOR = {Kubi\'{s}, Wies{\l}aw},
     TITLE = {Weak {F}ra\"{\i}ss\'{e} categories},
   JOURNAL = {Theory Appl. Categ.},
  FJOURNAL = {Theory and Applications of Categories},
    VOLUME = {38},
      YEAR = {2022},
     PAGES = {2, pp.\ 27--63},
     fPAGES = {Paper No. 2, 27--63},
      ISSN = {1201-561X},
   MRCLASS = {03C95 (18A30)},
  MRNUMBER = {4369354},
       DOI = {10.4208/cmr.2021-0061},
       URL = {https://doi.org/10.4208/cmr.2021-0061},
}

@article{MR1141931,
    AUTHOR = {Cameron, Peter J.},
     TITLE = {The age of a relational structure},
      fNOTE = {Directions in infinite graph theory and combinatorics
              (Cambridge, 1989)},
   JOURNAL = {Discrete Math.},
  FJOURNAL = {Discrete Mathematics},
    VOLUME = {95},
      YEAR = {1991},
    NUMBER = {1-3},
     PAGES = {49--67},
      ISSN = {0012-365X,1872-681X},
   MRCLASS = {03C15 (03C07 08A05 08A30 08A35)},
  MRNUMBER = {1141931},
MRREVIEWER = {G.\ Fuhrken},
       DOI = {10.1016/0012-365X(91)90329-Z},
       URL = {https://doi.org/10.1016/0012-365X(91)90329-Z},
}

@article{MR0069239,
    AUTHOR = {Fra\"{\i}ss\'{e}, Roland},
     TITLE = {Sur l'extension aux relations de quelques propri\'{e}t\'{e}s
              des ordres},
   JOURNAL = {Ann. Sci. \'{E}cole Norm. Sup. (3)},
  FJOURNAL = {Annales Scientifiques de l'\'{E}cole Normale Sup\'{e}rieure.
              Troisi\`eme S\'{e}rie},
    VOLUME = {71},
      YEAR = {1954},
     PAGES = {363--388},
      ISSN = {0012-9593},
   MRCLASS = {27.2X},
  MRNUMBER = {69239},
MRREVIEWER = {Djuro\ Kurepa},
DOI = {10.24033/asens.1027},
       fURL = {http://www.numdam.org/item?id=ASENS_1954_3_71_4_363_0},
}

@book{MR1221741,
    AUTHOR = {Hodges, Wilfrid},
     TITLE = {Model theory},
    SERIES = {Encyclopedia of Mathematics and its Applications},
    VOLUME = {42},
 PUBLISHER = {Cambridge University Press, Cambridge},
      YEAR = {1993},
     PAGES = {xiv+772},
      ISBN = {0-521-30442-3},
   MRCLASS = {03-01 (03-02 03Cxx)},
  MRNUMBER = {1221741},
MRREVIEWER = {J.\ M.\ Plotkin},
       DOI = {10.1017/CBO9780511551574},
       URL = {https://doi.org/10.1017/CBO9780511551574},
}

@misc{bartos2021weak,
      title={The weak {Ramsey} property and extreme amenability}, 
      author={Adam Bartoš and Tristan Bice and Keegan Dasilva Barbosa and Wiesław Kubiś},
      year={2021},
      eprint={2110.01694},
      archivePrefix={arXiv},
      fprimaryClass={math.LO}
}

@phdthesis{MR4478610,
    AUTHOR = {Dasilva Barbosa, Keegan},
    TITLE = {Ramsey {D}egree {T}heory of {O}rdered and {D}irected {S}ets},
    SCHOOL = {University of Toronto},
    YEAR = {2022},
    fTYPE = {Ph.D. Dissertation},
    fADDRESS = {Ann Arbor, MI},
    fPUBLISHER = {ProQuest LLC},
    fNOTE = {Available from ProQuest Dissertations & Theses Global. (28970102)},
    ISBN = {979-8834-07764-0},
    MRCLASS = {99-05},
    MRNUMBER = {4478610},
    URL = {https://www.proquest.com/docview/2700765015},
    fURL = {http://gateway.proquest.com/openurl?url_ver=Z39.88-2004&rft_val_fmt=info:ofi/fmt:kev:mtx:dissertation&res_dat=xri:pqm&rft_dat=xri:pqdiss:28970102},
    PAGES = {103}
}

@article{MR2308230,
    AUTHOR = {Kechris, Alexander S. and Rosendal, Christian},
     TITLE = {Turbulence, amalgamation, and generic automorphisms of
              homogeneous structures},
   JOURNAL = {Proc. Lond. Math. Soc. (3)},
  FJOURNAL = {Proceedings of the London Mathematical Society. Third Series},
    VOLUME = {94},
      YEAR = {2007},
    NUMBER = {2},
     PAGES = {302--350},
      ISSN = {0024-6115,1460-244X},
   MRCLASS = {03E15 (37B05)},
  MRNUMBER = {2308230},
MRREVIEWER = {Tam\'{a}s\ M\'{a}trai},
       DOI = {10.1112/plms/pdl007},
       URL = {https://doi.org/10.1112/plms/pdl007},
}

@article{MR4452123,
    AUTHOR = {Malicki, Maciej},
     TITLE = {Remarks on weak amalgamation and large conjugacy classes in
              non-archimedean groups},
   JOURNAL = {Arch. Math. Logic},
  FJOURNAL = {Archive for Mathematical Logic},
    VOLUME = {61},
      YEAR = {2022},
    NUMBER = {5-6},
     PAGES = {685--704},
      ISSN = {0933-5846,1432-0665},
   MRCLASS = {03E15 (54H11)},
  MRNUMBER = {4452123},
MRREVIEWER = {Cheng\ Chang},
       DOI = {10.1007/s00153-021-00807-1},
       URL = {https://doi.org/10.1007/s00153-021-00807-1},
}

@article{MR4458208,
    AUTHOR = {Panagiotopoulos, Aristotelis and Tent, Katrin},
     TITLE = {Universality vs genericity and {$C_4$}-free graphs},
   JOURNAL = {European J. Combin.},
  FJOURNAL = {European Journal of Combinatorics},
    VOLUME = {106},
      YEAR = {2022},
     PAGES = {103590, 12pp.},
     fPAGES = {Paper No. 103590, 12},
      ISSN = {0195-6698,1095-9971},
   MRCLASS = {03C15 (05C75)},
  MRNUMBER = {4458208},
MRREVIEWER = {Martin\ Weese},
       DOI = {10.1016/j.ejc.2022.103590},
       URL = {https://doi.org/10.1016/j.ejc.2022.103590},
}

@misc{drzewiecka2023generics,
      title={Generics in invariant subsets of some highly homogeneous permutation groups}, 
      author={Monika Drzewiecka and Aleksander Ivanov and Bartosz Mokry},
      year={2023},
      eprint={2303.07915},
      archivePrefix={arXiv},
      fprimaryClass={math.LO}
}

@article{Computable-Fraisse,
    AUTHOR = {Csima, Barbara F. and Harizanov, Valentina S. and Miller,
              Russell and Montalb\'{a}n, Antonio},
     TITLE = {Computability of {F}ra\"{i}ss\'{e} limits},
   JOURNAL = {J. Symbolic Logic},
  FJOURNAL = {The Journal of Symbolic Logic},
    VOLUME = {76},
      YEAR = {2011},
    NUMBER = {1},
     PAGES = {66--93},
      ISSN = {0022-4812,1943-5886},
   MRCLASS = {03D25 (03D30)},
  MRNUMBER = {2791338},
MRREVIEWER = {Roland\ Sh.\ Omanadze},
       DOI = {10.2178/jsl/1294170990},
       URL = {https://doi.org/10.2178/jsl/1294170990},
}
